\documentclass[a4paper,10pt]{article}

\usepackage[svgnames,dvipsnames,x11names]{xcolor}
\usepackage{amsmath}
\usepackage{amssymb}
\usepackage{bbm}
\usepackage{enumerate}
\usepackage{algorithm}
\usepackage{algorithmic}
\usepackage{amsfonts}
\usepackage{amsthm}
\usepackage{bm}
\usepackage[mathscr]{eucal}
\usepackage{a4wide}
\usepackage{authblk}
\usepackage{tikz}
\usetikzlibrary{decorations.pathreplacing}
\usepackage{tikz-cd}
\usepackage[colorlinks=true,allcolors=blue]{hyperref}

\newcommand{\Mcal}{\ensuremath{\mathcal{M}}}

\newcommand{\directlim}{\ensuremath{\underrightarrow{\mathrm{lim}} \,}}
\newcommand{\invlim}{\ensuremath{\underleftarrow{\mathrm{lim}} \,}}

\newcommand{\spc}[1]{\ensuremath{\mathbb{#1}}}      

\newcommand{\Xsp}{\spc{X}}                          
\newcommand{\Ysp}{\spc{Y}}                          
\newcommand{\Lsp}{\spc{L}}
        
\newcommand{\Rbb}{\mathbb{R}}
\newcommand{\Nbb}{\mathbb{N}}
\newcommand{\Sbb}{\mathbb{S}}

\newcommand{\Xcal}{\ensuremath{\mathcal{X}}} 
\newcommand{\Ycal}{\ensuremath{\mathcal{Y}}}

\newcommand{\Tcal}{\ensuremath{\mathcal{T}}}

\newcommand{\supp}{\mathrm{supp}}

\newcommand*{\swap}[2]{\hspace{-0.5ex}#2#1}

\allowdisplaybreaks

\newtheorem{theorem}{Theorem}[section]
\newtheorem{definition}{Definition}[section]
\newtheorem{lemma}{Lemma}[section]
\newtheorem{proposition}{Proposition}[section]
\newtheorem{corollary}{Corollary}[section]

\title{Projective limits of probabilistic symmetries\\and their applications to random graph limits}
\author[1]{Pim van der Hoorn}
\author[2,3,4]{Huck Stepanyants}
\author[3,4,5,6]{Dmitri Krioukov}

\affil[1]{Department of Mathematics and Computer Science, Eindhoven University of Technology}
\affil[2]{Department of Physics, Harvard University}
\affil[3]{Department of Physics, Northeastern University}
\affil[4]{Network Science Institute, Northeastern University}
\affil[5]{Department of Mathematics, Northeastern University}
\affil[6]{Department of Electrical and Computer Engineering, Northeastern University}

\date{}

\begin{document}

\maketitle

\begin{abstract}
  We couple projective limits of probability measures to direct limits of their symmetry groups. We show that the direct limit group is the group of symmetries of the projective limit probability measure. If projective systems of probability measures represent point processes in increasingly larger finite regions of the same infinite space, then we show that under some additional niceness and consistency assumptions, an extension of the direct limit group is the symmetry group of the projective limit point process in the whole infinite space. The application of these results to random graph limits provides ``shortest paths'' to graphons and graphexes as it recovers these random graph limits as trivial corollaries. Another application example encompasses a broad class of limits of random graphs with bounded average degrees. This class includes a representative collection of paradigmatic random graph models that have attracted significant research attention in diverse areas of science. Our approach thus provides a general unified framework to study limits of very different types of random graphs.
\end{abstract}

\newpage

\tableofcontents

\newpage

\section{Introduction}

Projective limits~\cite{mac2013categories,bourbaki2004theory,bourbaki1995topological} are a very general and powerful concept applicable to a wide variety of categories---sets, topological spaces, measurable spaces, or probability measures. Projective limits of probability measures can be seen as a generalization of the Kolmogorov extension theorem~\cite{kolmogoroff1933grundbegriffe}, which establishes the conditions for the existence of a stochastic process as a limit of a family of finite-dimensional distributions. A classic example is the Dirichlet process, extensively studied and used in statistics and machine learning~\cite{ferguson1973bayesian, orbanz2011projective, broderick2012beta, lin2010construction}. The projective limit in this case is the limit of Dirichlet distributions~\cite{orbanz2011projective}, used in the latent Dirichlet allocation~\cite{blei2003latent}, for instance.

Direct limits~\cite{mac2013categories,bourbaki2004theory,bourbaki1995topological} are as general and powerful, as they are dual to projective limits~\cite{mac2013categories,weibel1994introduction}. In fact, projective limits are also known as \emph{inverse} limits, to emphasize the duality. This duality boils down to inverting the direction of all the arrows in the commutative diagrams in the definitions of the limits. A classic example of this duality is the Pontryagin duality between direct limits of locally compact abelian groups and the inverse limits of their Pontryagin-dual groups, which are the groups of continuous homomorphisms from the original group to the circle group~\cite{morris1977pontryagin}. Yet the objects to which projective and direct limits apply do not have to belong to the same category. 

Here, we couple projective limits of probability measures to direct limits of groups with respect to which these measures are invariant. Our first result in Section~\ref{ssec:main_results} states that under certain natural conditions, the projective limit of probability measures is invariant with respect to the direct limit of groups of symmetries of these measures. That is to say that probabilistic symmetries are preserved in the limit, if things are right. This result is very general, and the closest previous results to this one, at least in spirit, appear to be~\cite{ackerman2016invariant, kleijn2025existence, olshanski1996ergodic, borodin2005harmonic, orbanz2011projective}.

However, in general, projective and direct limits deal with product spaces that are not always terribly useful in applications. Things change drastically if spaces that these limits deal with embed into each other, i.e., if a smaller space is a subspace of a bigger space. And this is exactly the settings we narrow down our next results to, motivated by applications. 

We consider projective systems of probability measures that represent point processes in increasingly larger finite chunks of the same infinite space. Our second result in Section~\ref{ssec:main_results} states that the projective limit of these measures in this case is a point process in the whole infinite space.

Finally, in our last result in Section~\ref{ssec:main_results}, we couple the projective systems of these point processes with the direct system of their symmetry groups, and show that the limit process in the whole infinite space is invariant with respect to not only the direct limit group, but also suitable larger groups.

These results, especially the last one, are strongly motivated by applications that we present in Section~\ref{sec:applications}. These applications deal with projective limits of random graphs~\cite{spencer2023projective}. We represent random graphs as point processes in which points are graph edges. If an edge connects vertices labeled by~$x$ and~$y$, where $x$ and~$y$ are elements of some label space~$\Lsp$, then this edge is a random point $\{x,y\}\in\Lsp\times\Lsp$, so the random graph is a point process in $\Lsp\times\Lsp$. A finite graph of size~$n$ has a finite label space~$\Lsp_n$, a bigger graph of size $m>n$ has a bigger~$\Lsp_m$, while an infinite graph needs an infinite~$\Lsp_\infty$. See Figure~\ref{fig} for an illustration.

It is now just a matter of choice what these label spaces and their symmetry groups are. Whatever they are, our last theorem in Section~\ref{ssec:main_results} tells us what the symmetry group in the limit is. We show in Section~\ref{sec:applications} that if one makes the most conventional choice of integer labels $\Lsp_n=[n]=\{1,2,\ldots,n\}$, Fig.~\ref{fig}(a), and takes the symmetry group to be the symmetric group of permutations of labels~$[n]$, then in the limit we get random graphs with labels in $\Lsp_\infty=\Nbb$ invariant with respect to the infinite symmetric group of permutations of~$\Nbb$. These graphs have a nice representation in terms of graphons~\cite{lovasz2012large,diaconis2007graph,janson2013graphons}. If one chooses instead the labels to be reals $\Lsp_n=[0,n]$, Fig.~\ref{fig}(b), and takes the symmetry group to the group of measure-preserving transformations of $[0,n]$, then in the limit we get random graphs with labels in $\Lsp_\infty=\Rbb_+$ invariant with respect to measure-preserving transformations of~$\Rbb_+$. These graphs have a nice representation in terms of graphexes~\cite{veitch2015class,borgs2019sampling,janson2022convergence}. In a way, these observations, based on applications of our coupling of projective and direct limits, provide ``shortest paths'' to graphons and graphexes.

Unfortunately, graphons are limits of dense graphs, graphexes are limits of sparse but not ultrasparse graphs, while most real-world networks are ultrasparse~\cite{boccaletti2006complex,barabasi2016network}. Here, by \emph{ultrasparse} we mean that the (expected) average degree is bounded in the limit, while by \emph{dense} we mean that it grows linearly with~$n$. For a variety of reasons, it has been a major struggle to come up with useful general limits for ultrasparse graphs. The most relevant and powerful notion of graph limits applicable to ultrasparse graphs and explored well in the past, is perhaps the Benjamini-Schramm limit, also known as local convergence~\cite{benjamini2011recurrence,aldous2004objective,hofstad2023local}. It characterizes well the local structure of the graph in the limit, yet it does not provide any prescription on how to sample a graph from its limit.

Our last example in Section~\ref{sec:applications} presents a simple combination of label spaces and their symmetry groups that can easily produce limits of ultrasparse graphs. This combination is actually the simplest out of the three examples presented in Section~\ref{sec:applications}. The label spaces~$\Lsp_n$ are concentric balls of volume~$n$ in~$\Rbb^d$, Fig.~\ref{fig}(c), while the symmetry group is the group of rotations in~$\Rbb^d$ around the center of the balls, so in the limit we get rotationally invariant graphs in the whole~$\Rbb^d$. This example includes a collection of impactful ultrasparse random graph models studied extensively in diverse domains---(soft) random geometric graphs~\cite{penrose2003random,penrose2016connectivity}, inhomogeneous random graphs~\cite{bollobas2007phase,chung2002average,hoorn2017sparse} and their geometric versions~\cite{bringmann2019geometric,krioukov2010hyperbolic,budel2024random}, as well as causal sets in quantum gravity~\cite{bombelli1987space,krioukov2012network,surya2025causal}. The key difference between this example and the other two above is that there is currently no---nice or ugly---representation result for these rotationally invariant random graphs in $\Rbb^d$. Obtaining such a result would put these ultrasparse graph limits on par with graphons and graphexes.

Forgoing this representational challenge, our approach that couples projective and direct limits, provides a flexible unified framework for studying limits of random graphs, be they dense, sparse, or ultrasparse. 

\begin{figure}
    \centerline{\includegraphics[width=\linewidth]{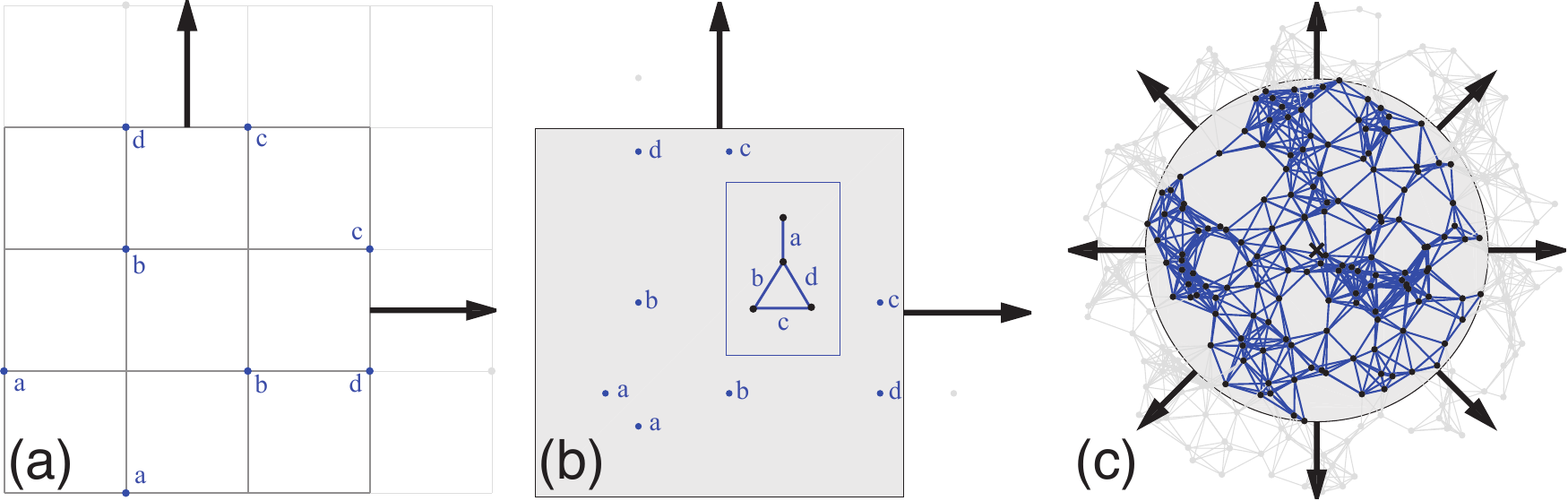}}
    \caption{ \textbf{Projective limits of random graphs as point processes.} \textbf{(a)}~A random graph of size~$4$ with vertices labeled by $\{1,2,3,4\}$. Each point represents an edge. The graph is undirected, so the point process is symmetric, with each edge represented by two symmetric points. The shown graph consists of four edges $a=\{1,2\}$, $b=\{2,3\}$, $c=\{3,4\}$, and $d=\{4,2\}$. Since the label space is~$\Nbb$ in the limit, the point process is confined to locations on the $\Nbb^2$ lattice. \textbf{(b)}~The same graph, shown in the inset, but with vertices labeled by real numbers in~$[0,4]$. Points can now be anywhere on the $[0,4]^2$ square, or on $\Rbb_+^2$ in the limit. \textbf{(c)}~A random geometric graph on a disk. The visualization is different compared to~(a) and (b) in that only a single copy of the label space is shown, the label space is two-dimensional, a disk in~$\Rbb^2$, points represent graph vertices versus edges, while edges are shown explicitly. The arrows in all the panels indicate the projective expansion of the label space as the graph grows.
    }
    \label{fig}
    \end{figure}

\section{Background information and definitions}\label{sec:background_results}

\subsection{Point processes}\label{ssec:point_process}

We begin with a brief recollection of the definition of point processes suitable for our needs~\cite{last2017lectures}. 

Let $\Xsp$ be a Polish space, with topology $T$ and denote by $\Xcal := \sigma(T)$ the associated Borel $\sigma$-algebra, and let $M(\Xsp)$ be the set of locally finite simple counting measures on $\Xsp$.  

We now endow $M(\Xsp)$ with its canonical measurable structure. For every $X \in \Xcal$ and $\xi \in M(\Xsp)$ we define the evaluation map $\epsilon_X : M(\Xsp) \to \mathbb N \cup \{\infty\}$ by
$\epsilon_X(\xi) := \xi(X)$, and let $\Mcal$ be the $\sigma$-algebra generated by these maps, $\Mcal := \sigma\bigl(\epsilon_X : X \in \Xcal\bigr)$. This is the standard $\sigma$-algebra on the space of locally finite simple counting measures used in the theory of point processes~\cite{last2017lectures}. We recall that $\Mcal$ is equal to the Borel $\sigma$-algebra generated by the \emph{vague topology} $\Tcal$ on $M(\Xsp)$, i.e. the topology generated by the integration maps $\pi_f: \xi \mapsto \int f \mathrm{d}\xi$ for $f \in C_c(\Xsp)$ which is the set of continuous functions $f$ on $\Xsp$ with bounded support~\cite[Theorem 4.2]{kallenberg2017random}. 

Now that we have a measurable space $(M(\Xsp), \Mcal)$, we define a \emph{point process} as a random element of $M(\Xsp)$. We identify such a point process with its corresponding probability measure $\mu$ on $(M(\Xsp), \Mcal)$, and say that $\mu$ is \emph{a point process on $\Xsp$}.

\subsection{Projective limits}\label{ssec:projective_limits}

Here we will briefly discuss the notions of projective limits of topological spaces and of measures on the corresponding Borel spaces. We consider topological spaces, and not general measurable spaces, as our space of interest $(M(\Xsp), \Mcal)$ is a Borel space with respect to the vague topology and this is an essential ingredient for studying limits of probability measures.

\subsubsection{Projective limits of topological spaces}\label{ssec:projective_limits_topospaces}

The main object in the framework of projective limits consists of a sequence of spaces together with \emph{projections} which is called a projective system. The notions of projective systems and projective limits are very general and can be defined for any category and can be indexed by any directed set~\cite{mac2013categories}. What we present below are more specific definitions for topological spaces indexed by positive integers that we will rely upon later on. We refer to~\cite[Chapter 4]{bourbaki1995topological} for more details.

\begin{definition}[Projective system of topological spaces~\cite{bourbaki1995topological}]\label{def:projective_system_top}
A \emph{projective system} of topological spaces consists of a collection $(\Xsp_n, T_n)_{n \ge 1}$ of topological spaces and a collection of continuous functions $\pi_{mn} : \Xsp_m \to \Xsp_n$ for all $n \le m$, such that $\pi_{nn} = \mathrm{id}_{\Xsp_n}$ and
\[
\begin{tikzcd}
	\Xsp_n  &\Xsp_m \arrow[l,swap,"\pi_{mn}"]\\
	&\Xsp_k \arrow[ul,"\pi_{kn}"] \arrow[u,swap,"\pi_{km}"]
\end{tikzcd}
\]
commutes for all $n \le m \le k$.
\end{definition}

We denote the projective system by $\langle \Xsp_n, T_n, \pi_{mn} \rangle_\Nbb$. The maps $\pi_{mn}$ are called \emph{projections}.

\begin{definition}[Projective limit of topological spaces~\cite{bourbaki1995topological}]\label{def:projective_limit_top}
Let $\langle \Xsp_n, T_n, \pi_{mn} \rangle_\Nbb$ be a projective system of topological spaces. A \emph{projective limit} of this system consist of a topological space $(\Xsp_\Nbb, T_\Nbb)$ and a collection of continuous functions $\pi_{n} : \Xsp_\Nbb \to \Xsp_n$ for $n \ge 1$, such that the following holds:
\begin{enumerate}
\item For every $n \le m$ the following diagram commutes
\begin{equation}\label{eq:def_projective_limit_projections}
\begin{tikzcd}
	\Xsp_n  &\Xsp_m \arrow[l,swap,"\pi_{mn}"]\\
	&\Xsp_\Nbb \arrow[ul,"\pi_{n}"] \arrow[u,swap,"\pi_{m}"]
\end{tikzcd}
\end{equation}
\item For every other topological space $(\Xsp, T)$ and continuous functions $\phi_n$ for which 
\[
\begin{tikzcd}
	\Xsp_n  &\Xsp_m \arrow[l,swap,"\pi_{mn}"]\\
	&\Xsp \arrow[ul,"\phi_n"] \arrow[u,swap,"\phi_m"]
\end{tikzcd}
\]
commutes for all $n \le m$, there exists a \emph{unique} continuous function $\phi : \Xsp \to \Xsp_\Nbb$ such that
\begin{equation}\label{eq:def_projective_limit_universal_diagram}
\begin{tikzcd}
	\Xsp_n &\Xsp_\Nbb \arrow[l,swap,"\pi_n"]\\
	&\Xsp \arrow[ul,"\phi_n"] \arrow[u,swap,"\phi"]
\end{tikzcd}
\end{equation}
commutes for all $n \ge 1$.
\end{enumerate} 
\end{definition}

The first property states that each smaller space $\Xsp_n$ can be obtained as a ``projection'' of the limit space $\Xsp_\Nbb$, in a manner that respects the local projections $\pi_{mn}$. The second property is referred to as the \emph{universal property}. 

We usually write $\langle \Xsp_\Nbb, T_\Nbb, \pi_n \rangle = \invlim \langle \Xsp_n, T_n, \pi_{mn} \rangle_\Nbb$ to denote that $\langle \Xsp_\Nbb, T_\Nbb, \pi_n \rangle$ is the projective limit of the system $\langle \Xsp_n, T_n, \pi_{mn} \rangle_\Nbb$.

An important standard result that we will use is that any projective system of topological spaces has a projective limit.
\begin{theorem}[Projective limit of topological spaces~\cite{bourbaki1995topological}]\label{thm:projective_limit_topological_spaces}
Let $\langle \Xsp_n, T_n, \pi_{mn}\rangle_\Nbb$ be a projective system of topological spaces. Let 
\[
	\Xsp_\Nbb := \left\{(y_n)_{n \ge 1} \in \prod_{n \ge 1} \Xsp_n \, : \, y_n = \pi_{mn}(y_m) \text{ for all } n \le m\right\}
\] 
and take $\pi_n : \Xsp_\Nbb \to \Xsp_n$ to be the canonical projections. Moreover, let $T_\Nbb$ to be the smallest topology that makes the canonical projections continuous, i.e., $T_\Nbb = T\left(\bigcup_{n \ge 1} \pi_n^{-1}(T_n)\right)$. Then
\[
	\langle \Xsp_\Nbb, T_\Nbb, \pi_n \rangle = \invlim \langle \Xsp_n, T_n, \pi_{mn} \rangle_\Nbb.
\]
\end{theorem}

\subsubsection{Projective limits of probability measures}
The notion of projective limits of topological spaces $(\Xsp_n, T_n)$ extends naturally to projective limits of probability measures. Here we assume that these spaces are Polish spaces equipped with their Borel sigma-algebras $\sigma(T_n)$ and probability measures on them. 

For this extension, one first needs to define what a projective system of probability measures is. For any measure $\mu$ on $(\Xsp, \Xcal)$ and measurable map $f: (\Xsp, \Xcal) \to (\Ysp, \Ycal)$, let $f \ast \mu$ be the push-forward of $\mu$ under $f$:
\[
	f \ast \mu(B) = \mu(f^{-1}(B)) \quad \text{for all } B \in \Ycal.
\]
\begin{definition}[Projective system of probability measures~\cite{bochner2005harmonic}]\label{def:projective_system_measures_pushforward}
Let $\langle\Xsp_n, T_n, \pi_{mn}\rangle_\Nbb$ be a projective system of Polish spaces. A collection $(\mu_n)_{n \ge 1}$ of probability measures on $\Xsp_n$ is called a \emph{projective system of measures} if $\mu_n = \pi_{mn} \ast \mu_m$ for all $n < m$, i.e., if the following diagram commutes for all $n < m$:
\[
\begin{tikzcd}
	\sigma(T_n) \arrow[r,"\pi_{mn}^{-1}"] \arrow[dr,swap,"\mu_n"] &\sigma(T_m) \arrow[d,"\mu_m"]  \\
	& \Rbb_+
\end{tikzcd}
\]
\end{definition}

The definition of a projective limit of probability measures is then as follows.
\begin{definition}[Projective limit of probability measures~\cite{bochner2005harmonic}]
Let $\langle \Xsp_n, T_n, \pi_{mn} \rangle_\Nbb$ be a projective system of Polish spaces and $(\mu_n)_{n \ge 1}$ be a projective sequence of probability measures with respect to $\pi_{mn}$. Then the \emph{projective limit of $(\mu_n)_{n \ge 1}$} (if it exists) is the unique probability measure $\mu_\Nbb$ on $\Xsp_\Nbb$ such that $\mu_n = \pi_n \ast \mu_\Nbb$ holds for all $n \ge 1$, where $\pi_n$ are the projections $\Xsp_\Nbb \to \Xsp_n$:
\[
\begin{tikzcd}
	\sigma(T_n) \arrow[r,"\pi_{n}^{-1}"] \arrow[dr,swap,"\mu_n"] &\sigma(T_\Nbb) \arrow[d,"\mu_\Nbb"]  \\
	& \Rbb_+
\end{tikzcd}
\]
\end{definition}

Similarly to projective systems of topological space, we will denote the projective limit of a sequence of probability measures $(\mu_n)_{n \ge 1}$ by $\mu_\Nbb = \invlim \mu_n$.

The following result, due to Bochner~\cite{bochner2005harmonic}, shows that any projective system of probability measures has a projective limit, which is a measure on the resulting projective limit space.

\begin{theorem}[Projective limit of probability measures~\cite{bochner2005harmonic}]\label{thm:projective_limit_prob_measures}
Let $\langle \Xsp_n, T_n, \pi_{mn} \rangle_\Nbb$ be a projective system of Polish spaces and let $(\mu_n)_{n \ge 1}$ be a projective sequence of probability measures on $(\Xsp_n, \sigma(T_n))$ with respect to $\pi_{mn}$. Then there exist a \emph{unique} probability measure $\mu_\Nbb$ on $(\Xsp_\Nbb, \sigma(T_\Nbb))$ such that $\mu_n = \pi_n \ast \mu_\Nbb$ for all $n \ge 1$, where $\pi_n$ are the canonical projections of the projective limit space. 
\end{theorem}

\subsection{Direct limits}

Our results will couple projective limits of probability measures to direct limits of groups acting on them. Direct limits are in a certain sense \emph{inverse} to projective limits, which are also known as \emph{inverse limits}. We refer to~\cite[Chapter 7]{bourbaki2004theory} or~\cite[Chapter 3]{mac2013categories} for a general treatment, while here we use direct limits in application to groups.

\begin{definition}[Direct system of groups~\cite{bourbaki2004algebraI}]\label{def:direct_system_group}
A \emph{direct system} of groups consists of a collection $(\Phi_n)_{n \ge 1}$ of groups and a collection of group homomorphisms $\eta_{nm} : \Phi_n \to \Phi_m$ for all $n \le m$, such that $\eta_{nn} = \mathrm{id}_{\Phi_n}$ and
\[
\begin{tikzcd}
	\Phi_n \arrow[r,"\eta_{nm}"] \arrow[dr, swap, "\eta_{nk}"] &\Phi_m \arrow[d,"\eta_{mk}"] \\
	&\Phi_k
\end{tikzcd}
\]
commutes for all $n \le m \le k$.
\end{definition}
We denote a direct system by $\langle \Phi_n, \eta_{nm} \rangle_\Nbb$.

When comparing Definition~\ref{def:direct_system_group} to Definition~\ref{def:projective_system_top} we indeed see that, apart from the fact that one is about groups and the other concerns topological spaces, the only difference is the directions of the arrows in the diagrams, which are inverted. Therefore, the definition of the direct limit of a direct system is analogous to that of a projective limit, with arrows inverted as well.
\begin{definition}[Direct limit of groups~\cite{bourbaki2004algebraI}]\label{def:direct_limit_group}
Let $\langle \Phi_n, \eta_{nm} \rangle_\Nbb$ be a direct system of groups. A \emph{direct limit} of this system consist of a group $\Phi_\Nbb$ and a collection of group homomorphisms $\eta_{n} : \Phi_n \to \Phi_\Nbb$ for $n \ge 1$, such that the following holds:
\begin{enumerate}
\item For every $n \le m$ the following diagram commutes
\begin{equation}\label{eq:def_direct_limit_maps}
\begin{tikzcd}
	\Phi_n \arrow[r,"\eta_{nm}"] \arrow[dr, swap, "\eta_{n}"]  &\Phi_m \arrow[d,"\eta_{m}"]\\
	&\Phi_\Nbb  
\end{tikzcd}
\end{equation}
\item For every other group $\Psi$ and group homomorphisms $\lambda_n$ for which 
\[
\begin{tikzcd}
	\Phi_n  \arrow[r,"\lambda_{nm}"] \arrow[dr,swap,"\lambda_n"] &\Phi_m \arrow[d,"\lambda_m"] \\
	&\Psi  
\end{tikzcd}
\]
commutes for all $n \le m$, there exists a \emph{unique} homomorphism $\lambda : \Phi_\Nbb \to \Psi$ such that
\begin{equation}\label{eq:def_direct_limit_universal_diagram}
\begin{tikzcd}
	\Phi_n \arrow[r,"\eta_n"] \arrow[dr,swap,"\lambda_n"] &\Phi_\Nbb \arrow[d,"\lambda"] \\
	&\Psi  
\end{tikzcd}
\end{equation}
commutes for all $n \ge 1$.
\end{enumerate} 
\end{definition}

We usually write $\langle \Phi_\Nbb, \eta_n \rangle = \directlim \langle \Phi_n, \eta_{nm} \rangle_\Nbb$ to denote that $\langle \Phi_\Nbb, \eta_n\rangle$ is the direct limit of the system $\langle \Phi_n, \eta_{nm} \rangle_\Nbb$.

As was the case with topological spaces, any direct system of groups has a direct limit. 
\begin{theorem}[Direct limit of groups~\cite{bourbaki2004algebraI}]\label{thm:direct_limit_groups}
Let $\langle \Phi_n, \eta_{nm}\rangle_\Nbb$ be a direct system of groups, denote by $\sqcup_{n \ge 1} \Phi_n$ the disjoint union of $\Phi_n$, and define on it the equivalence relation $\sim$ between $\varphi_n\in\Phi_n$ and $\varphi_m\in\Phi_m$ by
\[
	\varphi_n \sim \varphi_m \text{ if and only if } \exists k \geq \max\{n,m\} \text{ such that } \eta_{nk}(\varphi_n) = \eta_{mk}(\varphi_m).
\]
Now define the group
\[
	\Phi_\Nbb := \bigsqcup_{n \ge 1} \Phi_n / \sim
\]
with group operation $[\varphi_n] [\varphi_m] = [\eta_{nk}(\varphi_n)\eta_{mk}(\varphi_m)]$ with some $k \geq \max\{n,m\}$. Define also the homomorphisms $\eta_n : \Phi_n \to \Phi_\Nbb$ by $\eta_n(\varphi_n) = [\varphi_n]$. Then
\[
	\langle\Phi_\Nbb, \eta_n \rangle = \directlim \langle \Phi_n, \eta_{nm}\rangle_\Nbb.
\]
\end{theorem}

For our main goal in this paper, which is to couple the notions of projective and direct limits in a useful way, the definition of the direct limit above turns out to be a bit too strong for what we need, so we introduce the notion of a direct pre-limit of groups, where the universal diagram~\eqref{eq:def_direct_limit_universal_diagram} in Definition~\ref{def:direct_limit_group} is not required. 
\begin{definition}[Direct pre-limit of groups]\label{def:direct_prelimit_groups}
Let $\langle \Phi_n, \eta_{nm} \rangle_\Nbb$ be a direct system of groups. A \emph{direct pre-limit} of this system consist of a group $\Phi_\infty$ and a collection of group homomorphisms $\eta_{n} : \Phi_n \to \Phi_\infty$ for $n \ge 1$ such that the following diagram commutes for every $n \leq m$:
\begin{equation}\label{eq:def_direct_prelimit_maps}
\begin{tikzcd}
	\Phi_n \arrow[r,"\eta_{nm}"] \arrow[dr, swap, "\eta_{n}"]  &\Phi_m \arrow[d,"\eta_{m}"]\\
	&\Phi_\infty
\end{tikzcd}
\end{equation}
\end{definition}

\subsection{Compatible limits}\label{ssec:invariance}

Having defined the notion of projective limits of probability measures and direct limits of groups, we now move forward to our main goal by defining the idea of probabilistic symmetry in the limit.

Let $(\Xsp, \Xcal)$ be a measurable space and $\Phi$ a group acting on $\Xsp$. We say that $\Phi$ \emph{acts measurably on $\Xsp$} if $\varphi^{-1} A \in \Xcal$ holds for all $\varphi\in\Phi$ and $A \in \Xcal$. This definition of measurable action is different from the standard one, see for example~\cite[Chapter 1]{kallenberg2006foundations} or~\cite[Chapter 7]{kallenberg2006probabilistic}. This requires a measurable structure on the group as well, which we do not need.

\begin{definition}[Probabilistic symmetry] 
A probability measure $\mu$ on measurable space $\Xsp$ is \emph{invariant with respect to a group $\Phi$} that acts measurably on $\Xsp$ if 
\[
	\mu(A) = \mu(\varphi^{-1}A), \quad \forall \varphi \in \Phi \text{ and } A \in \Xcal.
\]
\end{definition}

The goal of this section is to define projective systems of invariant probability measures and their limits. This requires us to couple a projective system of topological spaces with a directed system of groups acting on them.

\begin{definition}[Compatible system of topological spaces and groups]\label{def:compatible_system_groups}
A projective system of topological spaces $\langle \Xsp_n, T_n, \pi_{mn} \rangle_\Nbb$ and a directed system of groups $\langle\Phi_n, \eta_{nm} \rangle_{\Nbb}$ acting measurably on $\Xsp_n$ are \emph{compatible} if
\[
	\pi_{mn}^{-1}(\varphi_n A) = (\eta_{nm} \varphi_n) \pi_{mn}^{-1}(A)
\]
holds for every $n \le m$, and every $\varphi_n \in \Phi_n$ and $A \in \Xcal_n = \sigma(T_n)$, i.e., if the following diagram commutes for all $n \le m$:
\begin{equation}\label{def:direct_system_compatible_diagram}
\begin{tikzcd}
	\Xsp_n \arrow[d,swap,"\varphi_n"] &\Xsp_m \arrow[l,swap,"\pi_{mn}"] \arrow[d,"\eta_{nm} \varphi_n"]  \\
	\Xsp_n &\Xsp_m \arrow[l,swap,"\pi_{mn}"]
\end{tikzcd}
\end{equation}
\end{definition}

We next add another ingredient to the picture, a projective system of probability measures. Let $\langle \Xsp_n, T_n, \pi_{mn} \rangle_\Nbb$ and $\langle \Phi_n, \eta_{nm} \rangle_\Nbb$ be a compatible system of Polish spaces and groups, and let, in addition, $\langle \Xsp_n, T_n, \pi_{mn} \rangle_\Nbb$ be equipped with a projective system of probability measures as in Definition~\ref{def:projective_system_measures_pushforward}. Then it follows immediately from Definition~\ref{def:compatible_system_groups} that
\begin{equation}\label{eq:direct_groups_commutes_with_projective_measures}
	\mu_n(\varphi_n A) = \pi_{mn} \ast \mu_m((\eta_{nm} \varphi_n) \pi_{mn}^{-1}(A))
\end{equation}
holds for all $n \le m$, $\varphi_n \in \Phi_n$, and $A \in \Xcal_n = \sigma(T_n)$. Informally, if we act on $A$ by $\varphi_n$ and measure what we get, then it is the same thing as sending both $A$ and $\varphi_n$ to their corresponding bigger $m$-spaces, measuring the result there, and then push-forward it back by projecting to the smaller $n$-space.

The final key ingredient that we will need to present our results in the next section is the definition of a direct pre-limit of groups that act compatibly on the projective limit of topological spaces. 
\begin{definition}[Compatible pre-limit group]\label{def:compatible_limit_group} 
Let $\langle \Xsp_n, T_n, \pi_{mn} \rangle_\Nbb$ be a projective system of topological spaces with projective limit $\langle \Xsp_\Nbb, T_\Nbb, \pi_n\rangle$ and $\langle \Phi_n, \eta_{nm}\rangle$ a direct system of groups acting on~$\Xsp_n$. A direct pre-limit $\langle \Phi_\infty, \eta_n\rangle$ is called \emph{compatible with $\pi_n$} if the following diagram commutes for all $n \ge 1$ and $\varphi_n \in \Phi_n$:
\begin{equation}\label{def:compatible_limit_group_diagram}
\begin{tikzcd}
	\Xsp_n \arrow[d,swap,"\varphi_n"] &\Xsp_\Nbb \arrow[l,swap,"\pi_{n}"] \arrow[d,"\eta_{n} \varphi_n"]  \\
	\Xsp_n &\Xsp_\Nbb \arrow[l,swap,"\pi_{n}"]
\end{tikzcd}
\end{equation}
i.e., if the limit homomorphisms $\eta_n : \Phi_n \to \Phi_\infty$ commute with the limit projections $\pi_n : \Xsp_\Nbb \to \Xsp_n$.
\end{definition}

\section{Main results: limits of probabilistic symmetries}\label{ssec:main_results}

We start with a general result, showing that projective limits of probability measures that are invariant with respect to a compatible direct system of groups, is invariant with respect to the direct limit of these groups.  

\begin{theorem}[Direct limit of compatible groups preserves invariance]\label{thm:direct_limit_compative_groups_invariance}
Let $\langle \Xsp_n, T_n, \pi_{mn}\rangle_\Nbb$ and $\langle \Phi_n, \eta_{nm} \rangle_\Nbb$ be a compatible projective system of Polish spaces and direct system of groups acting on $\Xsp_n$. Furthermore, let $\langle\Xsp_\Nbb, T_\Nbb, \pi_n\rangle$ and $\langle \Phi_\Nbb, \eta_n\rangle$ denote, respectively, their projective and direct limits. Then the group $\Phi_\Nbb$ acts on $(\Xsp_\Nbb, T_\Nbb)$ compatibly with respect to $\pi_n$, i.e., the diagram~\eqref{def:compatible_limit_group_diagram} commutes. 
Moreover, if $(\mu_n)_{n \ge 1}$ is a projective system of probability measures on $(\Xsp_n, \sigma(T_n))$ that are invariant with respect to $\Phi_n$, then the projective limit $\mu_\Nbb = \invlim \mu_n$ is invariant with respect to $\Phi_\Nbb = \directlim \Phi_n$.
\end{theorem}

The rest of our main results are about projective limits of point processes and their invariance. Specifically, we identify the symmetry group with respect to which the projective limit of a point process is invariant. 

Consider a projective system $(\mu_n)_{n \ge 1}$ of point processes on $\Xsp_n$ that are invariant under a direct system of groups. Then it follows from Theorem~\ref{thm:direct_limit_compative_groups_invariance} that the projective limit of these point processes $\mu_\Nbb$, which is a probability measure on the projective limit of the $\langle M(\Xsp_n), \sigma(T_n), \pi_{mn}\rangle_\Nbb$, is invariant under the direct limit of the groups. This is nice, but the main problem with this observation is that according to Theorem~\ref{thm:direct_limit_compative_groups_invariance}, the projective limit $\mu_\Nbb$ is a general probability measure on a projective limit space, while we actually want $\mu_\Nbb$ to be a point process on a given Polish space. For this to be the case, we need to impose some additional conditions that would allow us to show that the projective limit of $\langle M(\Xsp_n), \sigma(T_n), \pi_{mn}\rangle_\Nbb$ is a space of locally finite measures $M(\Xsp_\infty)$ on some space $\Xsp_\infty$.

To do so, let $(\Xsp, T)$ be a Polish space and $\Xcal$ the associated Borel $\sigma$-algebra. Next, consider an exhausting sequence $(\Xsp_n)_{n \ge 1}$ of non-decreasing open sets, i.e., $\Xsp_n \in T$, $\Xsp_n \subseteq \Xsp_m$ for any $n \le m$, and $\bigcup_{n \ge 1} \Xsp_n = \Xsp$. Then each $\Xsp_n$ with the induced subset topology is again a Polish space and we denote its Borel $\sigma$-algebra by $\Xcal_n$. In particular, we can now consider sequences $(\mu_n)_{n \ge 1}$ of point processes on $\Xsp_n$. 

To define the projective structure we denote by $\iota_{nm} : \Xsp_n \hookrightarrow \Xsp_m$ and $\iota_n : \Xsp_n \hookrightarrow \Xsp$ the canonical inclusions, and note that these are continuous. Then we can define the projections $\pi_{mn} : M(\Xsp_m) \to M(\Xsp_n)$, as a pull back, by
\begin{equation}\label{def:projections_mn}
	\pi_{mn}\xi (A_n) = \xi(\iota_{nm}(A_n)), 
\end{equation}
for $A_n$ a measurable set in $M(\Xsp_n)$. Similarly, we define the projections $\pi_{n} : M(\Xsp) \to M(\Xsp_n)$ by
\begin{equation}\label{def:projections_n}
	\pi_{n}\xi(A)  = \xi(\iota_{n}(A_n)).
\end{equation}

We now take a projective system of point processes on $\Xsp_n$ and show that their projective limit is a point process on $\Xsp$. 

\begin{theorem}[Projective limits of point processes exist]\label{thm:projective_limit_random_graphs}
Let $\Xsp$ and $(\Xsp_n)_{n \ge 1}$ be as above and let $(\mu_n)_{n \ge 1}$ be a projective system of point processes on $\Xsp_n $ with respect to $\pi_{mn}$ defined in~\eqref{def:projections_mn}. Then the projective limit $\mu_\Nbb = \invlim \mu_n$ exists and corresponds to a probability measure on $M(\Xsp)$, i.e., a point process on $\Xsp$, which is the unique probability measure satisfying $\pi_n \ast \mu_\Nbb = \mu_n$ for all $n \ge 1$, where $\pi_n$ is defined in~\eqref{def:projections_n}.
\end{theorem}

Having established the existence of projective limits of point processes on $\Xsp_n$, we now turn to their symmetries. To start, we extend the group action from $\Xsp_n$ to $M(\Xsp_n)$. Let $\Phi_n$ be a group that acts measurably on $\Xsp_n$. Then the group action on a measure $\xi \in M(\Xsp_n)$ is given by
\[
	\varphi_n \xi (A) := \xi(\varphi_n^{-1} A),
\]
where $\varphi_n\in\Phi_n$ and $A\in\Xcal_n$. We note that this action is measurable as well. 

Combining Theorem~\ref{thm:direct_limit_compative_groups_invariance} and Theorem~\ref{thm:projective_limit_random_graphs} now implies that a projective system of point processes $(\mu_n)_{n \ge 1}$ that is invariant under a direct system of compatible groups $\langle \Phi_n, \eta_{nm} \rangle_\Nbb$ acting on $\Xsp_n$ will yield a projective limit point process $\mu_\Nbb$ 
on $\Xsp$ that is invariant under the direct limit $\Phi_\Nbb$ of the groups. 

However, this result is not sufficiently strong for some applications, in particular those that we will consider in the next section. For such applications, we want to go beyond the direct limit group because this group turns out to be not the full limit group with respect to which we want to show invariance of our limit point process. The direct limit group is instead a dense finitary subgroup of this full group. To illustrate what we mean, let $\Phi_n$ be the group of permutation on $[n] := \{1, 2, \dots, n\}$. They form a direct system whose direct limit $\Phi_\Nbb$ is the group of permutations of $\Nbb$ that only change a finite number of elements~\cite[Chapter 6]{kegel2000locally}. However, we really want to show the invariance of our limit measure with respect to the group $\Phi_\infty$ of all permutations on $\Nbb$, and this is where our definition of a compatible direct pre-limit group comes into play.

\begin{theorem}[Projective limits of invariant point processes are invariant]\label{thm:projective_limit_invariant_random_graphs}
Let $(\mu_n)_{n \ge 1}$ be a projective system of point processes on $\Xsp_n$ with respect to $\pi_{mn}$, as in Theorem~\ref{thm:projective_limit_random_graphs}, and let $\mu_\Nbb = \invlim \mu_n$ be its projective limit. Let $(\Phi_n)_{n \ge 1}$ be a compatible directed system of groups acting on $\Xsp_n$, and let $\Phi_\infty$ be a direct pre-limit group acting on $\Xsp$, such that any $\mu$ on $M(\Xsp)$ is invariant with respect to $\Phi_\infty$ if 
\[
	\mu(A) = \mu((\eta_n \varphi_n)^{-1}A), \quad \forall n \ge 1, \, \varphi_n \in \Phi_n, \text{ and } A \in \Xcal.
\]
If $\mu_n$ is invariant with respect to $\Phi_n$ for every $n \ge 1$, then $\mu_\Nbb$ is invariant with respect to $\Phi_\infty$.
\end{theorem}

Although the requirements on the limit group $\Phi_\infty$ seem rather restrictive, we will provide a variety of relevant examples in the next section where they hold.

\section{Applications to limits of random graphs}\label{sec:applications}

In this section, we showcase the implications of our results by applying them to random graphs. We show how specific choices of spaces $\Xsp$ and $\Xsp_n$ yield limits of different types of random graphs. In the first two examples, we recover the known limits of dense graphs and sparse graphs known as \emph{graphons}~\cite{lovasz2012large,diaconis2007graph,janson2013graphons} and \emph{graphexes}~\cite{veitch2015class,borgs2019sampling,janson2022convergence}, which are invariant with respect to permutations of integer and real labels of vertices, respectively. The third example presents a completely new limit of ultrasparse graphs that are invariant under rotations, where vertices are labeled by their positions in~$\Rbb^d$.

\subsection{Random graphs as point processes}\label{ssec:graphs_point_process}

Random graphs can be viewed as point processes in spaces $\Xsp := \Lsp \times \Lsp$, where $\Lsp$ is the space of vertex labels, e.g., $\Lsp=\Nbb$ in the simplest case. Points in such a point process are graph edges.

Indeed, let $(\Lsp, T)$ be a Polish space and $\mathcal{L}$ its Borel $\sigma$-algebra. Then $\Xsp := \Lsp \times \Lsp$ with the product topology is again Polish. Denote by $M_s(\Xsp)$ the simple \emph{symmetric} counting measures on $\Xsp$, i.e., the counting measures $\xi$ that satisfy $\xi = \xi \sigma^{-1}$, where $\sigma(x,y) = (y,x)$. Any undirected (random) graph $G$ with vertex set $V \subset \Lsp$ and edge set $E \subset V \times V$ has an associated (random) measure $\xi^G \in M_s(\Xsp)$ given by
\[
	\xi^G = \sum_{(x,y) \in E} \delta_{(x,y)}.
\]
The isolated vertices labeled by~$z\in\Lsp$ in $G$ are represented by self-loop edge-points $\delta_{(z,z)}$. Conversely, any (random) symmetric measure $\xi = \sum_i \delta_{e_i}$, where $e_i\in\Xsp$, has an associated (random) undirected graph with the edge set $E=\cup_i e_i$ and vertex set $V = \{x \in \Lsp \, : \, \exists i \text{ and } y \in \Lsp \text{ such that } e_i = (x,y)\}$. Therefore, we will identify the space of symmetric measures $M_s(\Xsp)$ on space $\Xsp=\Lsp \times \Lsp$ with the space of undirected graphs and vertex labels in $\Lsp$. Technically, in order to have $\xi^G$ be locally finite we need that for any bounded set $B \subseteq \Lsp$ the set of vertices in $B$ is finite, i.e., $|B \cap V| < \infty$. This will be true for all the examples considered in this section, so we choose to skip the additional technical definitions here.

If graph $G$ is random, and so is its associated measure $\xi^G$, we will denote by $\mu$, as before, $\xi^G$'s probability measure on $M_s(\Xsp)$, and refer to this~$\mu$ as a random graph on $\Lsp$. We will also say that $(\mu_n)_{n \ge 1}$ is a projective system of random graphs with respect to some projections $\pi_{mn}$, if $(\mu_n)_{n \ge 1}$ is a projective system of symmetric point processes.

\subsection{Projective limits of random graphs}\label{sec:projective_graph_limits}

Here we show how our general results in Section~\ref{ssec:main_results} apply to general projective limits of random graphs.

To set up projective systems and limits of random graphs, let $(\Lsp_n)_{n \ge 1}$ be an exhausting sequence of non-decreasing open sets of the label space, i.e., $\Lsp_n \in T$, $\Lsp_n \subseteq \Lsp_m$ for any $n \le m$, and $\bigcup_{n \ge 1} \Lsp_n = \Lsp$. Then the spaces $\Xsp_n := \Lsp_n \times \Lsp_n$ with the product of the induced subset topologies are again Polish spaces. Moreover, if $\omega_{nm} : \Lsp_n \hookrightarrow \Lsp_m$ and $\omega_n : \Lsp_n \hookrightarrow \Lsp$ are the canonical inclusions, then the canonical inclusions $\iota_{nm} : \Xsp_n \hookrightarrow \Xsp_m$ and $\iota_n : \Xsp_n \hookrightarrow \Xsp$ are given by
\[
	\iota_{nm} = \omega_{nm} \otimes \omega_{nm}, \quad \text{and} \quad \iota_n = \omega_n \otimes \omega_n,
\]
where by $f \otimes f$ we mean the function that maps $(x,y) \in \Lsp \times \Lsp$ to $(f(x), f(y))$ for any function $f : \Lsp \to \Lsp$.

With these conventions, we are now in specific settings of Section~\ref{sec:background_results}. Therefore, this next corollary follows immediately from Theorem~\ref{thm:projective_limit_random_graphs}:
\begin{corollary}[Projective limits of random graphs exist]\label{cor:projective_limit_random_graphs}
Let $\Lsp$ and $(\Lsp_n)_{n \ge 1}$ be as described above and let $(\mu_n)_{n \ge 1}$ be a projective system of random graphs on $\Xsp_n := \Lsp_n \times \Lsp_n$ with respect to $\pi_{mn}$ defined in~\eqref{def:projections_mn}. Then the projective limit $\mu_\Nbb = \invlim \mu_n$ exists and corresponds to a random graph on $\Lsp$.
\end{corollary}

To apply Theorem~\ref{thm:projective_limit_invariant_random_graphs} to random graphs, we observe that any group $\Phi_n$ acting on $\Lsp_n$ can be trivially extended to a group $\Phi_n \otimes \Phi_n$ acting on $M(\Xsp_n)$ by defining for every measurable set $A \subset \Xsp_n$,
\[
	\varphi_n \xi (A) = \xi (\varphi_n^{-1} \otimes \varphi_n^{-1} (A)),
\]
where $\varphi_n \otimes \varphi_n (A) = \{(\varphi_n(x), \varphi_n(y)) \, : \, \text{ for all } (x,y) \in A\}$, and similarly for the group $\Phi_\infty$ acting on $\Lsp$. Moreover, any directed system $\langle \Phi_n, \eta_{nm}\rangle_\Nbb$ of groups, where $\Phi_n$ acts on $\Lsp_n$, yields a trivial extension to a directed system of groups $\langle \Phi_n \otimes \Phi_n, \eta_{nm} \otimes \eta_{nm} \rangle_\Nbb$, where $\Phi_n \otimes \Phi_n$ acts on $\Lsp_n \times \Lsp_n$, and where the product homomorphisms apply to the product group elements via
\[
	\eta_{nm} \otimes \eta_{nm}(\varphi_n \otimes \varphi_n^{\prime})  := \eta_{nm} \varphi_n \otimes \eta_{nm} \varphi_n^\prime
\] 
for any $\varphi_n,\varphi_n^\prime\in\Phi_n$.
To minimize clutter, we will slightly abuse the notation below by not distinguishing between the group acting on a label space and the same group acting on the space of random graphs on this label space, as they are related as above. We will similarly abuse the notation for homomorphisms $\eta_{nm}$ as well.

With these notations, this corollary follows immediately from Theorem~\ref{thm:projective_limit_invariant_random_graphs}:
\begin{corollary}[Projective limits of invariant random graphs are invariant]\label{cor:invariant_random_graphs}
Let $(\mu_n)_{n \ge 1}$ be a projective system of random graphs on $\Lsp_n$ with respect to $\pi_{mn}$, as in Theorem~\ref{thm:projective_limit_random_graphs}. Moreover, let $(\Phi_n)_{n \ge 1}$ and $\Phi_\infty$ be defined as in Theorem~\ref{thm:projective_limit_invariant_random_graphs}. Then, if each random graph $\mu_n$ is invariant with respect to $\Phi_n$ for every $n \ge 1$, the projective limit random graph $\mu_\Nbb = \invlim \mu_n$ is invariant with respect to $\Phi_\infty$.
\end{corollary}

The only difference between Corollaries~\ref{cor:projective_limit_random_graphs},~\ref{cor:invariant_random_graphs} and Theorems~\ref{thm:projective_limit_random_graphs},~\ref{thm:projective_limit_invariant_random_graphs}, apart from the nomenclature, is that our point processes must be symmetric if we want undirected graphs. However, we will see in the proofs in Section~\ref{sec:proofs} that the property of being symmetric is preserved when taking a projective limit: if we start with a projective system of symmetric point process, then the projective limit will also be symmetric.

To simplify the presentation of specific examples in the next section, we will rely on the following lemma, whose proof is also in Section~\ref{sec:proofs}:
\begin{lemma}\label{lem:nice_homomorphisms}
Let $\langle \Phi_n, \eta_{nm} \rangle_\Nbb$ be a directed system of groups and $\eta_n$ be homomorphisms $\Phi_n \to \Phi_\infty$ that make the group $\Phi_\infty$ a direct pre-limit.
\begin{enumerate}
\item The direct system $\langle \Phi_n, \eta_{nm} \rangle_\Nbb$ is compatible with the projections $\pi_{mn}$ defined in~\eqref{def:projections_mn} if the following diagram commutes for all $n \le m$ and $\varphi_n \in \Phi_n$
\[
	\begin{tikzcd}
		\Lsp_n \arrow[d,swap,"\varphi_n"] \arrow[r,"\iota_{mn}"] &\Lsp_m \arrow[d,"\eta_{nm} \varphi_n"]  \\
		\Lsp_n \arrow[r,"\iota_{mn}"] &\Lsp_m 
	\end{tikzcd}
\]
\item The group $\Phi_\infty$ is a compatible direct pre-limit with respect to $\pi_n$ as defined in~\eqref{def:projections_n} if the following diagram commutes for all $n \ge 1$ and $\varphi_n \in \Phi_n$: 
\[
	\begin{tikzcd}
		\Lsp_n \arrow[d,swap,"\varphi_n"] \arrow[r,"\iota_{n}"] &\Lsp \arrow[d,"\eta_{n} \varphi_n"]  \\
		\Lsp_n \arrow[r,"\iota_{n}"] &\Lsp 
	\end{tikzcd}
\]
\end{enumerate}
\end{lemma}

\subsection{Specific examples of projective graph limits}

Here we apply Corollaries~\ref{cor:projective_limit_random_graphs} and~\ref{cor:invariant_random_graphs} to three specific random graph limits, which are still quite general from the point of view of random graphs. The first two examples recover known powerful notions of limits of dense and sparse random graphs, while the third one reveals a new notion of graph limits applicable to a wide class of ultrasparse graphs, which were found useful in practical applications.

To describe each example properly, we will have to:
\begin{enumerate}
  \item specify:
\begin{enumerate}
\item the label spaces $\Lsp_n$ and $\Lsp$;
\item the symmetry groups $\Phi_n$ and $\Phi_\infty$, as well as their actions on $\Lsp_n$ and $\Lsp$, which extend to the product spaces as discussed in Section~\ref{sec:projective_graph_limits};
\item the group homomorphisms $\eta_{nm} : \Phi_n \to \Phi_m$ and $\eta_n : \Phi_n \to \Phi_\infty$;
\end{enumerate}
\item and verify that:
\begin{enumerate}
\item the space $\Lsp$ is Polish and $(\Lsp_n)_{n \ge 1}$ is a exhausting non-decreasing sequence of open sets;
\item the system $\langle \Phi_n, \eta_{nm} \rangle_\Nbb$ forms a compatible directed system of groups;
\item the group $\Phi_\infty$ is a compatible direct pre-limit, i.e., diagrams~\eqref{eq:def_direct_prelimit_maps} and~\eqref{def:compatible_limit_group_diagram} commute; and 
\item invariance with respect to each $\Phi_n$ implies invariance with respect to $\Phi_\infty$.
\end{enumerate}

\end{enumerate}
We will refer to the numbering above in our specific examples below.

\subsubsection{Graphons as projective limits of random graphs on $\Nbb$, invariant under permutations}

\begin{enumerate}
\item
\begin{enumerate}
\item We let $\Lsp = \Nbb$ and $\Lsp_n = [n] := \{1,2, \dots, n\}$, and endow each of these with the standard discrete topology.
\item We take $\Phi_n$ to be the finite symmetric group of permutations on $[n]$, and $\Phi_\infty$ the infinite symmetric group of permutations of $\Nbb$. 
\item We let $\eta_{nm} : \Phi_n \to \Phi_m$ and $\eta_n : \Phi_n \to \Phi_\infty$ be the canonical embeddings.
\end{enumerate}
\item
\begin{enumerate}
\item Since we deal with the standard discrete topology, $\Lsp$ is Polish as required. Moreover, we observe that each $\Lsp_n$ is indeed an open subset of $\Lsp$, and clearly $(\Lsp_n)_{n \ge 1}$ is an exhausting non-decreasing sequence.
\item Because we work with the canonical embeddings, it is clear that $(\Phi_n)_{n \ge 1}$ together with $\eta_{nm}$ forms a directed system of groups. In addition, since $\eta_{nm} \varphi_n$ denotes a permutation of $[m]$ that leaves $\{n+1, \dots, m\}$ fixed and permutes the other elements the same as it did in $[n]$, it follows that the conditions of point~1 of Lemma~\ref{lem:nice_homomorphisms} are satisfied. Hence, the system $\langle \Phi_n, \eta_{nm} \rangle_\Nbb$ is a compatible directed system of groups.
\item Similar to the previous point, the use of the canonical embeddings ensure that diagram~\eqref{eq:def_direct_prelimit_maps} commutes. The fact that $\eta_n \varphi_n$ as permutation on $\Nbb$ leaves $\{n+1, n+2, \dots\}$ fixed while permuting the others in the same way as it did in $[n]$, implies that the conditions of point~2 of Lemma~\ref{lem:nice_homomorphisms} are satisfied. This then implies that diagram~\eqref{def:compatible_limit_group_diagram} commutes.
\item It is known~\cite{diaconis2007graph} that a measure on $\Nbb \times \Nbb$ is invariant under permutations of $\Nbb$ if it is invariant under any permutation $\varphi_n \in \Phi_n$ for all $n \ge 1$.
\end{enumerate}
\end{enumerate}

Therefore, the application of Corollaries~\ref{cor:projective_limit_random_graphs} and~\ref{cor:invariant_random_graphs} to these specific settings results in that the projective limit of any projective system of random graphs on $[n]$ that are invariant under finite permutations is a random graph on $\Nbb$ that is invariant under infinite permutations. Furthermore, by invoking the classical representation result by Aldous and Hoover~\cite{aldous1981representations,hoover1979relations}, this limit random graph is a probabilistic mixture of $W$-random graphs, where $W : [0,1] \times[0,1] \to [0,1]$ is a symmetric measurable function known as a \emph{graphon}. A $W$-random graph is constructed by considering a sequence $x_1, x_2, \dots$ of i.i.d.\ points sampled from $U[0,1]$, and connecting vertices $i$ and $j$ independently with probability $W(x_i, x_j)$. 

We summarize these observations in the following corollary:
\begin{corollary}[Graphons as projective limits]
Let $(G_n)_{n \ge 1}$ be a projective sequence of random graphs with labels in $[n]$ that are invariant under permutations. Then the projective limit $G_\infty = \invlim G_n$ is a random graph with labels in $\Nbb$ that is invariant under infinite permutations. Therefore, $G_\infty$ is a $W$-random graph, where $W$ is a (possibly random) graphon.
\end{corollary}

\subsubsection{Graphexes as projective limits of random graphs on $\Rbb_+$, invariant under measure-preserving transformations}

\begin{enumerate}
\item
\begin{enumerate}
\item We let $\Lsp = \Rbb_+$ and $\Lsp_n = [0,n)$, and endow $\Rbb_+$ with the standard Euclidean topology.
\item We take $\Phi_n$ to be the group of measure-preserving transformation of $[0,n)$,  and $\Phi_\infty$ the group of measure-preserving transformations of $\Rbb_+$.
\item Any measure-preserving transformation of $[0,n)$ yields a measure-preserving transformation of $[0,m)$ or $\Rbb_+$ by keeping $[n,m)$ or $[n,\infty)$ fixed. We let $\eta_{nm} : \Phi_n \to \Phi_m$ and $\eta_n : \Phi_n \to \Phi_\infty$ denote the corresponding canonical embeddings.
\end{enumerate}
\item
\begin{enumerate}
\item Since we work with the Euclidean topology, $\Lsp$ is Polish as required. Moreover, we observe that each $\Lsp_n$ is indeed an open subset of $\Lsp$, and clearly $(\Lsp_n)_{n \ge 1}$ is an exhausting non-decreasing sequence.
\item Because we deal with with the canonical embeddings, it is clear that $(\Phi_n)_{n \ge 1}$ together with $\eta_{nm}$ forms a directed system of groups. In addition, the conditions of point~1 of Lemma~\ref{lem:nice_homomorphisms} are satisfied. Hence, the system $\langle \Phi_n, \eta_{nm} \rangle_\Nbb$ is a compatible directed system of groups.
\item Similar to the previous point, the use of the canonical embeddings ensure that diagram~\eqref{eq:def_direct_prelimit_maps} commutes and that the conditions of point~2 of Lemma~\ref{lem:nice_homomorphisms} are satisfied. It then follows that diagram~\eqref{def:compatible_limit_group_diagram} commutes.
\item A random measure on $\Rbb_+ \times \Rbb_+$ is invariant under measure-preserving transformations of $\Rbb_+$ if and only if it is invariant under transpositions of dyadic intervals, see Proposition~9.1 in~\cite{kallenberg2006probabilistic}. Since any two such intervals are finite and thus contained in $[0,m]$ for some $m$, it follows that invariance with respect to $\Phi_n$ for all $n \ge 1$ implies invariance with respect to $\Phi_\infty$.
\end{enumerate}
\end{enumerate}

Therefore, the application of Corollaries~\ref{cor:projective_limit_random_graphs} and~\ref{cor:invariant_random_graphs} to these specific settings results in that the projective limit of any projective system of random graphs on $[0,n)$ that are invariant under measure-preserving transformations is a random graph on $\Rbb_+$ that is invariant under measure-preserving transformations. Furthermore, by invoking the representation result~\cite[Theorem 2.8]{borgs2019sampling} or \cite[Theorem 4.7]{veitch2015class}), which is a direct consequence of a general representation result for jointly exchangeable random measures on $\Rbb_+\times\Rbb_+$ by Kallenberg~\cite{kallenberg2006probabilistic}, this limit random graph is a probabilistic mixture of $\mathcal{W}$-random graphs, where $\mathcal{W}$ is a graphex. A \emph{graphex} is a triple $\mathcal{W} = (I, S, W)$, where $I \in \Rbb_+$, $S : \Rbb_+ \to \Rbb_+$ is measurable with $S \wedge 1$ being integrable, and $W : \Rbb_+ \times \Rbb_+ \to \Rbb_+$ is a symmetric measurable function with some additional conditions that we do not go into for brevity. To sample a $\mathcal{W}$-random graph on $[0,n)$, one samples the unit-rate Poisson point process~$(x_i,y_i)$ on $[0,n)\times\Rbb_+$, creates edges between vertices labeled by $x_i$ and $x_j$ independently with probability $W(y_i,y_j)$, and then removes zero-degree vertices. The $I$ and $S$ components of the graphex can be ignored. See the references above for further details. 

We summarize these observations in the following corollary:
\begin{corollary}[Graphexes as projective limits]
Let $(G_n)_{n \ge 1}$ be a sequence of random graphs with labels in $[0,n)$ that are invariant under measure-preserving transformations. Then the projective limit $G_\infty = \invlim G_n$ is a random graph with labels in $\Rbb_+$ that is invariant under measure-preserving transformations. Therefore, $G_\infty$ is a graphex-random graph with a (possibly random) graphex $\mathcal{W} = (I, S, W)$. 
\end{corollary}

\subsubsection{Projective limits of random graphs on $\Rbb^d$, invariant under rotations}

While the previous two examples stayed in the domain of known results for dense and sparse graphs, this last example moves us to the new domain of projective limits of ultrasparse graphs.

\begin{enumerate}
\item
\begin{enumerate}
\item Fix $d \ge 2$ and let $\Lsp$ be $\Rbb^d$ endowed with the standard Euclidean topology. Take $\Lsp_n = \mathcal{B}_n^d$ to be the open concentric $d$-dimensional balls of volume $n$.
\item Let $\Phi_n = \Phi_\infty = \Phi$ be the group of rotations of $\Rbb^d$ around the center of the balls. The action of $\Phi_n$ is then simply the rotation restricted to the ball $\Lsp_n$.
\item Since all our groups are the same, their embeddings are trivial identities.
\end{enumerate}
\item
\begin{enumerate}
\item Since we deal with the Euclidean topology, $\Lsp$ is Polish as required. Moreover, we observe that each $\Lsp_n$ is indeed an open subset of $\Lsp$, and clearly $(\Lsp_n)_{n \ge 1}$ is an exhausting non-decreasing sequence.
\item This follows immediately from the fact that all groups are the same.
\item This follows immediately from the fact that all groups are the same.
\item This follows immediately from the fact that all groups are the same.
\end{enumerate}
\end{enumerate}

Therefore, the application of Corollaries~\ref{cor:projective_limit_random_graphs} and~\ref{cor:invariant_random_graphs} to these specific settings results in that the projective limit of any projective system of random graphs on $\mathcal{B}^d_n$ that are invariant under rotations is a random graph on $\Rbb^d$ that is invariant under rotations.

Unfortunately, in contrast with the previous two examples, we lack any representation results for these graphs. However, since their symmetry group is small and does not increase with $n$, it is clear that the class of such graphs is extremely vast. In particular, it includes graphs that can be obtained by first sampling a rotationally invariant point process in $\Rbb^d$, and then connecting vertices $i$ and $j$ labeled by their spherical coordinates $(r_i,\theta_i)$ and $(r_j,\theta_j)$ independently with probability $W(r_i,r_j,\theta_{ij})$, where $\theta_{ij}$ is the spherical distance between $i$ and $j$ on the $(d-1)$-dimensional unit sphere $\Sbb^{d-1}$, and $W : \Rbb_+\times\Rbb_+\times\Sbb^{d-1}\to[0,1]$ is any function. We assume that the center of the balls is at the origin.

This subclass of rotationally invariant graphs is also very broad. It includes many random graph models that attracted significant research attention in different domains. First, the most basic examples of rotationally invariant point processes in $\Rbb^d$ are the homogeneous Poisson point processes, but the general representation result is also known~\cite[Theorem 4.1.2]{bryc1995normal}. It says that, as expected, the radial coordinates of points can be sampled independently from any measure on $\Rbb_+$, while the angular coordinates must be independent and uniform on~$\Sbb^{d-1}$.

More importantly, by selecting different classes of functions $W$, we recover different popular random graph models. Here are some prominent examples:
\begin{itemize}
  \item If $W$ does not depend on anything at all, i.e., if it is a constant function, then we recover the Erd\H{o}s-R\'enyi graphs~\cite{erdos1959random}.
  \item If $W$ depends only on the distance between two vertices in $\Rbb^d$ or $\Sbb^{d-1}$, then we recover (soft) random geometric graphs in $\Rbb^d$ or $\Sbb^{d-1}$, respectively~\cite{penrose2003random,penrose2016connectivity}.
  \item If $W$ depends only on the sum of the radial coordinates of two vertices, then we recover inhomogeneous random graphs~\cite{bollobas2007phase}, i.e., random graphs with given expected degree sequences~\cite{chung2002average} or distributions~\cite{hoorn2017sparse}.
  \item If $W$ depends on the hyperbolic distance between two vertices, then we recover geometric inhomogeneous random graphs~\cite{bringmann2019geometric}, i.e., random hyperbolic graphs~\cite{krioukov2010hyperbolic,budel2024random}.
  \item If $W$ depends only on the radial and angular distances between two vertices in a certain way, then we recover causal sets in quantum gravity~\cite{bombelli1987space,krioukov2012network,surya2025causal}.
\end{itemize}
In all these examples, if the point process is of a finite rate, and if $W$ is integrable, then graphs are ultrasparse.

\section{Proofs}\label{sec:proofs}

Here we present the proofs of our main results. Recall that $(\Xsp, T)$ is a Polish space and $\Xcal$ denotes the associated Borel $\sigma$-algebra. 

\subsection{Theorem~\ref{thm:direct_limit_compative_groups_invariance}: Direct limit of compatible groups preserves invariance}\label{ssec:proof_direct_limit_invariance}

We will first prove Theorem~\ref{thm:direct_limit_compative_groups_invariance} which shows that the projective limit of probability measures, which are invariant under a direct system of groups, is invariant under the direct limit of these groups. One of the key notions that will be used in the proof is that of compatibility of the maps $\eta_{nm} : \Phi_n \to \Phi_m$ and $\eta_n : \Phi_n \to \Phi_\Nbb$, with the projections $\pi_{mn}$ and $\pi_n$, respectively. We recall the commuting diagrams for these notions, where $n \le m$, see also Definition~\ref{def:compatible_system_groups} and Definition~\ref{def:compatible_limit_group}.
\[
	\begin{tikzcd}
		\Xsp_n \arrow[d,swap,"\varphi_n"] &\Xsp_m \arrow[l,swap,"\pi_{mn}"] \arrow[d,"\eta_{nm} \varphi_n"]  \\
		\Xsp_n &\Xsp_m \arrow[l,swap,"\pi_{mn}"]
	\end{tikzcd}~\hspace{20pt}
	\begin{tikzcd}
		\Xsp_n \arrow[d,swap,"\varphi_n"] &\Xsp_\Nbb \arrow[l,swap,"\pi_{n}"] \arrow[d,"\eta_{n} \varphi_n"]  \\
		\Xsp_n &\Xsp_\Nbb \arrow[l,swap,"\pi_{n}"]
	\end{tikzcd}
\]
We also write $\varphi_n(A) = \{\varphi_n x \, : \, x \in A\}$ to denote the action of an element $\varphi_n \in \Phi_n$ on a set $A \subset \Xsp_n$.

\begin{proof}[Proof of Theorem~\ref{thm:direct_limit_compative_groups_invariance}]
Recall the definition of $\Xsp_\Nbb$ given in Theorem~\ref{thm:projective_limit_topological_spaces} and that of $\Phi_\Nbb$ in Theorem~\ref{thm:direct_limit_groups}. We first define the action of $\Phi_\Nbb$ on $\Xsp_\Nbb$. 

\paragraph{Action of $\Phi_\Nbb$}
For $[\varphi_k] \in \Phi_\Nbb$ and $\vec{x} \in \Xsp_\Nbb$ we define $[\varphi_k] \vec{x} = \vec{y}$ with
\begin{equation}\label{eq:def_direct_group_action}
	y_m = \begin{cases}
		\pi_{km} (\varphi_k x_k) &\text{if } m \le k,\\
		(\eta_{km} \varphi_k) x_m &\text{if } m > k.
	\end{cases}
\end{equation}

We will first show that $\vec{y} \in \Xsp_\Nbb$, which means that $\pi_{mn} y_m = y_n$ holds for all $n \le m$. So let $n \le m$ and assume that $k < n$. Then
\begin{align*}
	\pi_{mn}y_m &= \pi_{mn}(\eta_{km} \varphi_k)x_m \\
	&= (\eta_{kn}\varphi_k) \pi_{mn} x_m \\
	&= (\eta_{kn}\varphi_k) x_n \\ 
	&= y_n.
\end{align*}
Here, the crucial second equality follows from the fact that the direct system $\langle\Phi_n, \eta_{nm}\rangle_\Nbb$ is compatible with the projective system $\langle \Xsp_n, T_n, \pi_{mn}\rangle_\Nbb$, diagram~\eqref{def:compatible_limit_group_diagram}.

If $n \le k < m$ we have
\begin{align*}
	\pi_{mn}y_m &= \pi_{mn}(\eta_{km} \varphi_k)x_m \\
	&= \pi_{kn} (\pi_{mk} (\eta_{km} \varphi_k) x_m) \\
	&= \pi_{kn} \varphi_k x_k \\
	&= y_n,
\end{align*}
where the third equality follows from the compatibility of the direct and projective system, diagram~\eqref{def:compatible_limit_group_diagram}.

Finally, if $n \le m \le k$ we get
\begin{align*}
	\pi_{mn} y_m &= \pi_{mn} \pi_{km} (\varphi_k x_k) = \pi_{kn} (\varphi_k x_k) = y_n.
\end{align*}
We conclude that $\pi_{mn} y_m = y_n$ for all $n \le m$ and thus that $\vec{y} = [\varphi_k] \vec{x} \in \Xsp_\Nbb$.

\paragraph{Action is well-defined}
We also need to show that the action is well-defined. This means we need to check that $e_\Nbb\vec{x} = \vec{x}$, with $e_\Nbb$ the identity in $\Phi_\Nbb$ and $([\varphi_k] [\varphi_\ell]) \vec{x} = [\varphi_k] ([\varphi_\ell] \vec{x})$. 

From the definition of $\Phi_\Nbb$ it follows that $e_\Nbb = [e_n]$ for any $n \ge 1$, where $e_n$ is the identity in $\Phi_n$. Looking at definition of the action in~\eqref{eq:def_direct_group_action}, if $\vec{y} = e_\Nbb \vec{x}$ it follows immediately that $y_m = x_m$ for all $m \ge 1$ and thus $e_\Nbb \vec{x} = \vec{x}$. We are left to prove the second part. 

Without loss of generality we assume that $k \le \ell$. Now take $m \ge 1$, and assume for now that $m \le k \le \ell$. Then, since $([\varphi_\ell] \vec{x})_m = \pi_{\ell m} \varphi_\ell x_\ell$ for all $m \le \ell$, we have
\begin{align*}
	([\varphi_k]([\varphi_\ell] \vec{x}))_m &= \pi_{km} \varphi_k([\varphi_\ell] \vec{x})_k \\
	&= \pi_{km} \varphi_k \pi_{km} \varphi_\ell x_\ell \\
	&= \pi_{km} \pi_{\ell k} \eta_{k\ell}(\varphi_k) \varphi_\ell x_\ell\\
	&= \pi_{\ell m} (\eta_{k\ell}(\varphi_k) \varphi_\ell x_\ell)\\
	&= (([\varphi_k] [\varphi_\ell]) \vec{x})_m, 
\end{align*}
where we used compatibility in the form of diagram~\eqref{def:direct_system_compatible_diagram} for the third equality.

When $k < m \le \ell$ we still have $([\varphi_\ell] \vec{x})_m = \pi_{\ell m} \varphi_\ell x_\ell$. Thus
\begin{align*}
	([\varphi_k]([\varphi_\ell] \vec{x}))_m &= (\eta_{km} \varphi_k)([\varphi_\ell] \vec{x})_m\\
	&=  (\eta_{km} \varphi_k)\pi_{\ell m} \varphi_\ell x_\ell \\
	&= \pi_{\ell m} \eta_{k \ell}(\varphi_k) \varphi_\ell x_\ell) \\
	&= (([\varphi_k] [\varphi_\ell]) \vec{x})_m,
\end{align*}
with compatibility, diagram~\eqref{def:direct_system_compatible_diagram}, yielding the third equality.

Finally, when $k \le \ell < m$ we have $([\varphi_\ell] \vec{x})_m = (\eta_{\ell m} \varphi_\ell) x_m$ and therefore,
\begin{align*}
	([\varphi_k]([\varphi_\ell] \vec{x}))_m &= (\eta_{km} \varphi_k)([\varphi_\ell] \vec{x})_m\\
	&= (\eta_{km} \varphi_k) (\eta_{\ell m} \varphi_\ell) x_m \\
	&= (\eta_{\ell m} (\eta_{k \ell} \varphi_k)) (\eta_{\ell m} \varphi_\ell) x_m \\
	&= \eta_{\ell m } (\eta_{k\ell} (\varphi_k) \varphi_\ell) x_m \\
	&= (([\varphi_k] [\varphi_\ell]) \vec{x})_m.
\end{align*}
Here we use that $\eta_{\ell m}$ is a group homomorphism for the third equality.

We can now conclude that the action of $\Phi_\Nbb$ on $\Xsp_\Nbb$ is indeed well-defined.

\paragraph{The action is compatible}
Next we show that the $\Phi_\Nbb$ acts on $\Xsp_\Nbb$ in a compatible way with respect to $\pi_n$, i.e. that the diagram~\eqref{def:compatible_limit_group_diagram} commutes. For this let $n \ge 1$ and take $\varphi_n \in \Phi_n$. Then, since $\eta_n \varphi_n = [\varphi_n]$ we get
\[
	\pi_n((\eta_n \varphi_n)\vec{x}) = \pi_n([\varphi_n]\vec{x}) = ([\varphi_n] \vec{x})_n
	= \varphi_n x_n = \varphi_n (\pi_n \vec{x}),
\]
and thus the diagram commutes.

\paragraph{The projective limit is invariant with respect to $\Phi_\Nbb$}

The final thing we need to show is that the projective limit measure $\mu_\Nbb$ is invariant with respect to $\Phi_\Nbb$. For this let $\varphi \in \Phi_\Nbb$ and define the measure $\hat{\mu}_\Nbb$ on $\Xsp_\Nbb$ by $\hat{\mu}_\Nbb (A) := \mu_\Nbb(\varphi^{-1} A)$. Since the measure $\mu_\Nbb$ is the unique measure such that $\pi_n \ast \mu_\Nbb = \mu_n$ holds for all $n \ge 1$, to prove invariance it suffices to show that $\pi_n \ast \hat{\mu}_\Nbb = \mu_n$.

So let us fix $n \ge 1$. We then note that by definition of the direct limit group it holds that $\varphi = [\varphi_m] = \eta_m \varphi_m$ for some $m \ge 1$. Taking $k = \max\{n,m\}$ we get that $\varphi = [\eta_k(\eta_{mk} \varphi_m)]$ and note that, since $\eta_k$ and $\eta_{mk}$ are homomorphisms, we have $\hat{\mu}_\Nbb (A) = \mu_\Nbb(\eta_k(\eta_{mk} \varphi_m^{-1}) A)$. We now compute that

\begin{align}
	\pi_n \ast \hat{\mu}_\Nbb(A_n) &= \hat{\mu}_\Nbb( \pi_n^{-1} A_n) \nonumber\\
	&= \hat{\mu}_\Nbb(\pi_k^{-1}(\pi_{kn}^{-1} A_n)) \label{eq:direct_limit_proof_pi} \\
	&= \mu_\Nbb(\eta_k(\eta_{mk} \varphi_m^{-1}) \pi_k^{-1}(\pi_{kn}^{-1} A_n)) \nonumber \\
	&= \mu_\Nbb(\pi_k^{-1} \eta_{mk} \varphi_m^{-1} \pi_{kn}^{-1}A_n) \label{eq:direct_limit_proof_compatible} \\
	&= \pi_k \ast \mu_\Nbb(\eta_{mk} \varphi_m^{-1} \pi_{kn}^{-1}A_n) \label{eq:direct_limit_proof_pi_k} \\
	&= \mu_k(\eta_{mk} \varphi_m^{-1} \pi_{kn}^{-1}A_n) \label{eq:direct_limit_proof_mu_k}\\
	&= \mu_k(\pi_{kn}^{-1} A_n) \label{eq:direct_limit_proof_invariance}\\
	&= \mu_n(A_n). \nonumber
\end{align} 

Let us explain the key steps in this computation. Step~\eqref{eq:direct_limit_proof_pi} follows from the fact that $\pi_k = \pi_n \pi_{kn}$. For step~\eqref{eq:direct_limit_proof_compatible} we use the fact that the action of the direct limit is compatible. In particular we applied diagram~\eqref{def:compatible_limit_group_diagram} for $\pi_k$ and the group element $\eta_{mk} \varphi_m^{-1}$. Next, we used the definition of the push forward in~\eqref{eq:direct_limit_proof_pi_k} and the fact that $\pi_k \ast \mu_\Nbb = \mu_k$ in~\eqref{eq:direct_limit_proof_mu_k}. The final important step~\eqref{eq:direct_limit_proof_invariance} is because $\eta_{km} \varphi_m \in \Phi_k$ and $\mu_k$ is invariant with respect to $\Phi_k$. 

All together, we conclude that indeed $\pi_n \ast \hat{\mu}_\Nbb = \mu_n$ holds for any $n \ge 1$ and thus that $\mu_\Nbb$ is invariant under $\Phi_\Nbb$.

\end{proof}

\subsection{Theorem~\ref{thm:projective_limit_random_graphs}: Projective limits of point processes exist}\label{ssec:proof_projective_limit_point_process}

From now on we are in the specific setting where $(\Xsp_n)_{n \ge 1}$ is an exhaustive non-decreasing sequence of open sets of $\Xsp$. We denote the associated subspace topology and $\sigma$-algebra by $T_n$ and $\Xcal_n$, respectively.

The main object of interest is the space $M(\Xsp)$ of locally finite simple (symmetric)
counting measures on $\Xsp$, equipped with the vague topology $\Tcal$, i.e., the
coarsest topology such that the maps
$\xi \longmapsto \int f \, d\xi$, $f \in C_c(X)$,
are continuous, where $C_c(X)$ is the space of continuous functions on $X$ with compact support.
We denote by $\Tcal_n$ the subspace topology on $M(\Xsp_n)$ induced by $\Tcal$.
Recall that $\Mcal$ is the $\sigma$-algebra on $M(\Xsp)$ introduced in Section~\ref{ssec:point_process} and that this coincides with the Borel $\sigma$-algebra of the vague topology~\cite{kallenberg2017random}.

The proof of Theorem~\ref{thm:projective_limit_random_graphs} consists of two main steps. First, we use the general constructions from Section~\ref{ssec:projective_limits_topospaces} to build the projective limit of the system $\langle M(\Xsp_n), \Tcal_n, \pi_{mn} \rangle_\Nbb$ and prove that this space is homeomorphic to $(M(\Xsp), \Tcal)$.

For the second step, we apply Theorem~\ref{thm:projective_limit_prob_measures} to obtain a projective limit measure for any projective sequence of probability measures $\mu_n$ on $(M(\Xsp_n), \Mcal_n)$, where $\Mcal_n$ is the subspace
$\sigma$-algebra on $M(\Xsp_n)$. Since the projective limit space is homeomorphic to $(M(\Xsp), \Tcal)$, we then get a probability measure on this space, which establishes the result.

We first show that in our setting the projective limit is homeomorphic to $(M(\Xsp), \Tcal)$. Recall that we defined the projections $\pi_n : M(\Xsp) \to M(\Xsp_n)$ via the canonical maps $\iota_n : \Xsp_n \to \Xsp$ as
\[
	\pi_n \xi (A_n) = \xi( \iota_n(A_n)).
\]

\begin{proposition}\label{prop:projective_limits_topology_isomorphic}
Let $\langle M(\Xsp_n), \Tcal_{n}, \pi_{mn} \rangle_\Nbb$ be a projective system of locally finite simple (symmetric) counting measures, with $\pi_{mn}$ defined as in~\eqref{def:projections_mn}, and let $\langle M_\Nbb, \Tcal_\Nbb, \phi_n \rangle$ be its projective limit as defined in Theorem~\ref{thm:projective_limit_topological_spaces}. Then there exists an isomorphism of topological spaces
\[
	h : (M_\Nbb, \Tcal_\Nbb) \to (M(\Xsp), \Tcal).
\]
Moreover, the isomorphism $h$ makes the following diagram commute for each $n \ge 1$
\[
	\begin{tikzcd}
		M(\Xsp_n) & M_\Nbb \arrow[l,swap,"\phi_n"] \arrow[dl, "h"] \\
		M(\Xsp) \arrow[u,"\pi_{n}"] &
	\end{tikzcd}
\]
\end{proposition}

\begin{proof}
Recall from Theorem~\ref{thm:projective_limit_topological_spaces} that the projective limit $\langle M_\Nbb, \mathcal{T}_\Nbb, \phi_n \rangle$ is defined by
\[
	M_\Nbb = \left\{\vec{\xi} := (\xi_n)_{n \ge 1} \, : \, \xi_n = \pi_{mn} \xi_m \text{ for all } n \le m\right\},
\]
with maps $\phi_n$ defined as $\phi_n(\vec{\xi}) = \xi_n$ and $\mathcal{T}_\Nbb = T\left(\bigcup_{n \ge 1} \phi_n^{-1}(\Tcal_{n})\right)$.

We will explicitly construct a continuous function $h : M_\Nbb \to M(\Xsp)$, show that it makes the diagram commute and that it has a continuous inverse $h^{-1} : M(\Xsp) \to M_\Nbb$. Because most of the work concerns the map $h$, we start the proof there.

\paragraph{The function $h$.} 

Let $A \subset \Xsp$ be a measurable set, define $A_n^\Xsp := A \cap \Xsp_n$ and note that $A_n^\Xsp \in \Xcal_n$. Recall that $\Xsp_n \subseteq \Xsp_m$ for $m \ge n$. Hence $\iota_{nm}(A_n^\Xsp) \subseteq A \cap \Xsp_m = A_m^\Xsp$ and therefore
\[
	\xi_n(A_n^\Xsp) = \pi_{mn}\xi_m(A_n^\Xsp) = \xi_m(\iota_{nm}(A_n^\Xsp)) \le \xi_m(A_m^\Xsp).
\]
In particular, $\xi_n(A_n^\Xsp)$ is a non-decreasing sequence. Thus, for any $\vec{\xi} \in M_\Nbb$ we can define $h(\vec{\xi})$ as
\begin{equation}
	h(\vec{\xi})(A) = \lim_{n \to \infty} \xi_n(A_n^\Xsp).
\end{equation}

\noindent\textit{The function $h$ maps to $M(\Xsp)$}\\
First we need to prove that $h(\vec{\xi})$ is indeed a measure. For each $n\ge1$ define a measure
$\bar\xi_n$ on $(\Xsp,\Xcal)$ by
\[
\bar\xi_n(A):=\xi_n(A\cap \Xsp_n), \qquad A\in\Xcal.
\]
Clearly each $\bar\xi_n$ is a (simple counting) measure on $\Xsp$. Moreover, for $m\ge n$ and
$A\in\Xcal$ we have, using $\Xsp_n\subseteq \Xsp_m$ and the projective consistency
$\xi_n=\pi_{mn}\xi_m$,
\[
\bar\xi_n(A)=\xi_n(A\cap \Xsp_n)=\pi_{mn}\xi_m(A\cap \Xsp_n)=\xi_m(\iota_{nm}(A\cap \Xsp_n))
\le \xi_m(A\cap \Xsp_m)=\bar\xi_m(A).
\]
Hence $(\bar\xi_n(A))_{n\ge1}$ is non-decreasing for every $A\in\Xcal$, and by definition
\[
h(\vec{\xi})(A)=\lim_{n\to\infty}\bar\xi_n(A), \qquad A\in\Xcal.
\]
We now verify that $h(\vec{\xi})$ is a measure. Let $(A_i)_{i\ge1}\subseteq \Xcal$ be pairwise disjoint.
For each $n$ we have $\bar\xi_n\!\left(\bigcup_{i\ge1}A_i\right)=\sum_{i\ge1}\bar\xi_n(A_i)$, and therefore
\[
h(\vec{\xi})\!\left(\bigcup_{i\ge1}A_i\right)
=\lim_{n\to\infty}\bar\xi_n\!\left(\bigcup_{i\ge1}A_i\right)
=\lim_{n\to\infty}\sum_{i\ge1}\bar\xi_n(A_i).
\]
Let $S_{n,M}:=\sum_{i=1}^M \bar\xi_n(A_i)$. Then $S_{n,M}$ is non-decreasing in both $n$ and $M$, and
$\sum_{i\ge1}\bar\xi_n(A_i)=\sup_{M\ge1}S_{n,M}$. Using monotonicity in both indices we obtain
\[
\lim_{n\to\infty}\sum_{i\ge1}\bar\xi_n(A_i)
=\lim_{n\to\infty}\sup_{M\ge1}S_{n,M}
=\sup_{M\ge1}\lim_{n\to\infty}S_{n,M}
=\sup_{M\ge1}\sum_{i=1}^M \lim_{n\to\infty}\bar\xi_n(A_i)
=\sum_{i\ge1} h(\vec{\xi})(A_i),
\]
so $h(\vec{\xi})$ is countably additive. In particular, $h(\vec{\xi})$ is a measure on $\Xsp$.

Finally, we need to show that $h(\vec{\xi})$ is locally finite. For this let $A \in \Xcal$ be a bounded
measurable set. By the exhaustion property of $(\Xsp_n)_{n\ge1}$ there exists $N\in\Nbb$ such that
$A\subseteq \Xsp_N$, and hence
\[
h(\vec{\xi})(A)=\lim_{n\to\infty}\xi_n(A\cap \Xsp_n)=\xi_N(A)<\infty,
\]
since $\xi_N$ is locally finite. We thus conclude that $h(\vec{\xi})\in M(\Xsp)$.

\noindent\textit{The function $h$ is continuous.}\\
Let $(\vec{\xi}^k)_{k \ge 1}$ be a sequence in $M_\Nbb$ such that $\vec{\xi}^k \to \vec{\xi}^\infty$ in
$(M_\Nbb,\mathcal{T}_\Nbb)$. By definition of the projective limit topology $\mathcal{T}_\Nbb$ (the initial
topology w.r.t.\ the maps $\phi_n$), we have, for every fixed $n \ge 1$,
\[
\xi_n^k = \phi_n(\vec{\xi}^k)\ \longrightarrow\ \phi_n(\vec{\xi}^\infty) = \xi_n^\infty
\qquad \text{in } (M(\Xsp_n),\Tcal_n).
\]
To prove that $h(\vec{\xi}^k) \to h(\vec{\xi}^\infty)$ in $(M(\Xsp),\mathcal{T})$ with $\mathcal{T}$ the
vague topology, it suffices to show that
\[
\int_{\Xsp} f \, d\, h(\vec{\xi}^k)\ \longrightarrow\ \int_{\Xsp} f \, d\, h(\vec{\xi}^\infty)
\qquad \text{for all } f \in C_c(\Xsp).
\]
Fix $f \in C_c(\Xsp)$. Since $\supp(f)$ is compact and $(\Xsp_n)_{n\ge1}$ exhausts $\Xsp$, there exists
$N$ such that $\supp(f) \subseteq \Xsp_N$. For any $\vec{\xi} \in M_\Nbb$, the measure $h(\vec{\xi})$
agrees with $\xi_N$ on $\Xsp_N$ (indeed, for $A \subseteq \Xsp_N$ we have $A \cap \Xsp_n = A$ for all
$n \ge N$ and $(\pi_{nN}\xi_n)(A)=\xi_N(A)$, hence $\xi_n(A)=\xi_N(A)$). It follows then that
\[
\int_{\Xsp} f \, d\, h(\vec{\xi}) = \int_{\Xsp_N} f \, d\xi_N .
\]
Applying this to $\vec{\xi}^k$ and $\vec{\xi}^\infty$ gives
\[
\int_{\Xsp} f \, d\, h(\vec{\xi}^k) = \int_{\Xsp_N} f \, d\xi_N^k,
\qquad
\int_{\Xsp} f \, d\, h(\vec{\xi}^\infty) = \int_{\Xsp_N} f \, d\xi_N^\infty.
\]
Since $\xi_N^k \to \xi_N^\infty$ in $(M(\Xsp_N),\Tcal_N)$ and $\Tcal_N$ is the vague topology on
$M(\Xsp_N)$, we obtain
\[
\int_{\Xsp_N} f \, d\xi_N^k \ \longrightarrow\ \int_{\Xsp_N} f \, d\xi_N^\infty,
\]
and therefore $\int f\, d\,h(\vec{\xi}^k) \to \int f\, d\,h(\vec{\xi}^\infty)$ for all $f \in C_c(\Xsp)$.
This is exactly $h(\vec{\xi}^k) \to h(\vec{\xi}^\infty)$ in the vague topology on $M(\Xsp)$, so $h$ is
continuous.

\noindent\textit{The function $h$ makes the diagram commute.}\\
Fix $n \ge 1$ and take $B_n \in \Xcal_n$. Then, since for all $k \ge n$ we have that $\iota_n(B_n) \cap \Xsp_k = \iota_n (B_n)$, it follows that
\begin{align*}
	\pi_n h(\vec{\xi})(B_n) &= h(\vec{\xi})(\iota_n(B_n)) \\
	&= \lim_{k \to \infty} \xi_k (\iota_n(B_n) \cap \Xsp_k) \\
	&= \xi_n(B_n) = \phi_n(\vec{\xi})(B_n).
\end{align*}

\paragraph{The inverse of $h$.} 

Define the map $g : M(\Xsp) \to M_\Nbb$ by
\[
	g(\xi) := (\xi|_{\Xsp_n})_{n \ge 1} = (\pi_n(\xi))_{n \ge 1},
\]
where $\xi|_{\Xsp_n}$ is the restriction of $\xi$ to $\Xsp_n$, i.e. $\xi|_{\Xsp_n}(A) = \xi(A \cap \Xsp_n)$.

First we observe that since the diagram
\[
	\begin{tikzcd}
		\Xsp \\
		\Xsp_n \arrow[u, "\iota_n"] \arrow[r, swap, "\iota_{nm}"]
		&\Xsp_m \arrow[ul, swap, "\iota_m"]
	\end{tikzcd}
\]
commutes, it follows that 
\[
\begin{tikzcd}
	M(\Xsp) \arrow[d, swap,"\pi_n"] \arrow[dr,"\pi_m"]\\
	M(\Xsp_n) &M(\Xsp_m) \arrow[l,"\pi_{mn}"]
\end{tikzcd}
\]
also commutes.

Hence, since $g(\xi)_n = \pi_n \xi$ we have that $\pi_{mn} g(\xi)_m = \pi_{mn} \pi_m \xi = \pi_n \xi = g(\xi)_n$ which shows that $g(\xi)$ indeed maps to $M_\Nbb$.

To see that $g = h^{-1}$ we first note that
\[
	gh(\vec{\xi}) = (\pi_n h(\vec{\xi}))_{n \ge 1} = (\phi_n(\vec{\xi}))_{n \ge 1} = (\xi_n)_{n \ge 1} = \vec{\xi},
\]
where we used that $\pi_n h = \phi_n$.

On the other hand, for any $A \in \Xcal$,
\[
	hg(\xi)(A) = \lim_{n \to \infty} g(\xi)_n(A \cap \Xsp_n)
	= \lim_{n \to \infty} \pi_n \xi(A \cap \Xsp_n)
	= \lim_{n \to \infty} \xi(\iota_n(A \cap \Xsp_n))
	= \xi(A)
\]
where we used that $\iota_n(A \cap \Xsp_n) \uparrow A$. We thus conclude that $hg(\xi) = \xi$ and hence $g$ is indeed the inverse of $h$.

It now remains to show that $g$ is continuous. Since $\mathcal{T}_\Nbb$ is the initial topology
with respect to the maps $\phi_n$, it suffices to show that $\phi_n \circ g$ is continuous for each
$n\ge1$. By the commuting diagram we have $\phi_n g = \pi_n$, where $\pi_n$ is the restriction map
$\pi_n(\xi)=\xi|_{\Xsp_n}$.

We claim that $\pi_n : (M(\Xsp),\mathcal{T}) \to (M(\Xsp_n),\Tcal_n)$ is continuous for the vague
topologies. Indeed, let $f\in C_c(\Xsp_n)$ and choose $\tilde f \in C_c(\Xsp)$ such that
$\tilde f|_{\Xsp_n}=f$. We can extend $f$ by $0$ outside $\Xsp_n$: since 
$\supp(f)\subseteq\Xsp_n$ is compact in $\Xsp$, the extension by $0$
is continuous and compactly supported because $\Xsp_n$ is open in $\Xsp$.
Then for all $\xi\in M(\Xsp)$,
\[
\int_{\Xsp_n} f \, d\,\pi_n(\xi) \;=\; \int_{\Xsp_n} f \, d(\xi|_{\Xsp_n})
\;=\; \int_{\Xsp} \tilde f \, d\xi.
\]
Since $\xi\mapsto \int_{\Xsp} \tilde f\, d\xi$ is continuous on $(M(\Xsp),\mathcal{T})$, it follows that
$\xi\mapsto \int_{\Xsp_n} f\, d\,\pi_n(\xi)$ is continuous, and hence $\pi_n$ is continuous.

Now let $A \in \bigcup_{n \ge 1} \phi_n^{-1}(\Tcal_n)$, so that $A=\phi_n^{-1}(A_n)$ for some
$A_n\in\Tcal_n$. Using $\phi_n g=\pi_n$ and continuity of $\pi_n$ we obtain
\[
g^{-1}(A) = g^{-1}(\phi_n^{-1}(A_n)) = (\phi_n g)^{-1}(A_n) = \pi_n^{-1}(A_n) \in \mathcal{T}.
\]
Therefore $g$ is continuous, and the proof is complete.
\end{proof}

Having established that the projective limit of $\langle M(\Xsp_n), \Tcal_n, \pi_{mn} \rangle_\Nbb$ is the correct space $(M(\Xsp), \Tcal)$, we turn to studying projective sequences of probability measures $(\mu_n)_{n \ge 1}$ on $M(\Xsp_n)$. Recall that due to Theorem~\ref{thm:projective_limit_prob_measures} any projective sequence of probability measures has a projective limit which is a probability measure on the projective limit space.

\begin{proof}[Proof of Theorem~\ref{thm:projective_limit_random_graphs}]
Let use denote by $\Mcal_n$ the subspace $\sigma$-algebra of $M(\Xsp_n)$. Then by Theorem~\ref{thm:projective_limit_prob_measures} any projective sequence $(\mu_n)_{n \ge 1}$ of probability measures on $(M(\Xsp_n), \Mcal_n)$ has a projective limit $\hat{\mu}_\infty$ which is a probability measure on the projective limit space $(M_\Nbb, \sigma(\Tcal_\Nbb))$, defined as in Theorem~\ref{thm:projective_limit_topological_spaces}. By Proposition~\ref{prop:projective_limits_topology_isomorphic} this space is homeomorphic to $(M(\Xsp), \Tcal)$ via the isomorphism $h$ and hence the probability measure $\hat{\mu}_\infty$ can be pushed forward to a probability measure $\mu_\Nbb := h \ast \hat{\mu}_\infty$ on $(M(\Xsp), \sigma(\Tcal))$, which equals $(M(\Xsp), \Mcal)$, and hence $\mu_\Nbb$ is a point process on~$\Xsp$.

Now take any $A_n \in \Mcal_n$ and let $g = h^{-1}$ be the inverse of the isomorphism $h$. Then, since $h \pi_n = \phi_n$ implies that $\pi_n = g \phi_n$ we have that 
\begin{align*}
	\pi_n \ast \mu_\Nbb(A_n) &= \mu_\Nbb(\pi_n^{-1} A_n) \\
	&= \hat{\mu}_\infty( g(\pi_n^{-1} A_n) ) \\
	&= \hat{\mu}_\infty(\phi_n^{-1} A_n) = \mu_n(A_n).
\end{align*}
Moreover, because by Theorem~\ref{thm:projective_limit_prob_measures} $\hat{\mu}_\infty$ was the unique probability measure such that $\phi_n \ast \hat{\mu}_\infty = \mu_n$, $\mu_\Nbb$ is the unique probability measure such that $\pi_n \ast \mu_\Nbb = \mu_n$.
\end{proof}

We end this section with the proof of Corollary~\ref{cor:projective_limit_random_graphs}, which uses the explicit construction of the function $h$ in Proposition~\ref{prop:projective_limits_topology_isomorphic}.

\begin{lemma}[Projective limits preserve symmetry]\label{lem:symmetry_is_preserved}
Theorem~\ref{thm:projective_limit_prob_measures} remains true if we replace $M(\Xsp \times \Xsp)$ with the subspace of symmetric counting measures $M_s(\Xsp \times \Xsp)$.
\end{lemma}

\begin{proof}
Carefully following the steps in the proof of Theorem~\ref{thm:projective_limit_prob_measures} we see that the one thing missing is to proof that the map $h$ encountered in the proof of Proposition~\ref{prop:projective_limits_topology_isomorphic} maps to symmetric point process on $\Lsp \times \Lsp$. So let  $A \subset \Lsp \times \Lsp$ be a measurable set and note that $\vec{\xi} = (\xi_n)_{n \ge 1}$ consists of symmetric counting measures $\xi_n$. Therefore, with $\sigma$ denoting the switching map and using the definition for $h$ given in the proof of Proposition~\ref{prop:projective_limits_topology_isomorphic} we get
\begin{align*}
	h(\vec{\xi})(\sigma^{-1} A) &= \lim_{n \to \infty} \xi_n(\sigma^{-1}A_n \cap \Lsp_n \times \Lsp_n) \\
	&= \lim_{n \to \infty} \xi_n(\sigma^{-1}(A \cap \Lsp_n \times \Lsp_n)) \\
	&= \lim_{n \to \infty} \xi_n(A \cap \Lsp_n \times \Lsp_n) = h(\vec{\xi}).
\end{align*}  
Here we used that $\sigma^{-1}A_n \cap \Lsp_n \times \Lsp_n = \sigma^{-1}(A \cap \Lsp_n \times \Lsp_n)$ and that each $\xi_n$ is symmetric.
\end{proof}

\subsection{Theorem~\ref{thm:projective_limit_invariant_random_graphs}: Projective limits of invariant point processes are invariant}\label{ssec:proofs_invariance}

We need to show that the projective limit $\mu_\Nbb$ is invariant with respect to the compatible direct pre-limit $\Phi_\infty$.

Recall that we have extended the action of $\Phi_n$ from $\Xsp_n$ to $M(\Xsp_n)$ by defining $\varphi_n \xi_n (A) = \xi_n(\varphi_n^{-1}(A))$ and similarly for the action of $\Phi_\infty$ on $M(\Xsp)$. 

Note that to prove invariance, it suffice to show that $\mu_\Nbb$ is invariant with respect to $\eta_m \varphi_m$ for any $m \ge 1$ and $\varphi_m \in \Phi_m$. For this we follow a strategy similar to the proof of invariance in Theorem~\ref{thm:direct_limit_compative_groups_invariance}, making use of the fact that the projective limit measure $\mu_\Nbb$ is the unique probability measure such that $\pi_n \ast \mu_\Nbb = \mu_n$.

\begin{proof}[Proof of Theorem~\ref{thm:projective_limit_invariant_random_graphs}]
Fix a $m \ge 1$, $\varphi_m \in \Phi_m$ and define $\hat{\mu}_\infty = \mu_\Nbb \circ (\eta_m \varphi_m)^{-1}$. Then we need to show that $\hat{\mu}_\infty = \mu_\Nbb$, and by the uniqueness of projective limits it suffices to show that $\pi_n * \hat{\mu}_\infty = \mu_n$ holds for all $n \ge 1$. 
Let $k = \max \{m,n\}$ and $A_n \in \Mcal_n$ be any measurable set of graphs with node labels in $\Xsp_n$. Then,
\begin{align}
    \pi_n \ast \hat{\mu}_\infty (A_n) &= \hat{\mu}_\infty(\pi_n^{-1} A_n) \nonumber \\
    &= \hat{\mu}_\infty( \pi_k^{-1} \pi_{kn}^{-1}(A_n)) \nonumber \\
    &= \mu_\Nbb((\eta_m \varphi_m)^{-1} \pi_k^{-1} \pi_{kn}^{-1}(A_n)) \nonumber \\
    &= \mu_\Nbb((\eta_m \varphi_m^{-1}) \pi_k^{-1} \pi_{kn}^{-1}(A_n)) \nonumber \\
    &= \mu_\Nbb((\eta_k \eta_{mk} \varphi_m^{-1}) \pi_k^{-1} \pi_{kn}^{-1}(A_n)) \label{eq:step_directed_system} \\
    &= \mu_\Nbb(\pi_k^{-1} (\eta_{mk} \varphi_m^{-1}) \pi_{kn}^{-1} A_n) \label{eq:step_eta_n_commute} \\
    &= \mu_k((\eta_{mk} \varphi_m)^{-1} \pi_{kn}^{-1} A_n) \label{eq:step_projection_k} \\
    &= \mu_k(\pi_{kn}^{-1} A_n) \label{eq:step_finite_symmetry} \\
    &= \mu_n(A_n) \label{eq:step_projective_system}.
\end{align}

For sack of clarity, let us elaborate on the validity of each step in the derivation above. The first three steps follow from the definition of the projections and $\hat{\mu}_\infty$ while the fourth step is because $\eta_m$ is a homomorphism. Step~\eqref{eq:step_directed_system} uses the commuting diagram~\eqref{eq:def_direct_prelimit_maps} which implies that $\eta_m \varphi_m = \eta_k \eta_{mk} \varphi_m$. In the next step~\eqref{eq:step_eta_n_commute}, we used that  that $\eta_{mk} \varphi_m \in \Phi_k$ while $\eta_k \varphi_k \pi_k^{-1} = \pi_k^{-1} \varphi_k$ hold for all $\varphi_k \in \Phi_k$ (see diagram~\eqref{def:compatible_limit_group_diagram}). After that, step~\eqref{eq:step_projection_k} follows from the fact that $\mu_\Nbb$ is the projective limit and hence $\pi_k \ast \mu_\Nbb = \mu_k$. Next, in~\eqref{eq:step_finite_symmetry} we use that $\eta_{mk} \varphi_m \in \Phi_k$ and each $\mu_k$ is invariant with respect to $\Phi_k$. The final step~\eqref{eq:step_projective_system} is due to fact that $(\mu_n)_{n \ge 1}$ is a projective system.

\end{proof} 

\subsection{Lemma~\ref{lem:nice_homomorphisms}}\label{ssec:proof_lemma_nice_homomorphism}

We first observe that the commuting diagrams in point~1 and point~2 imply that the following diagrams commute, respectively,
\[
	\begin{tikzcd}
		\Lsp_n \times \Lsp_n \arrow[d,swap,"\varphi_n \otimes \varphi_n"] \arrow[r,"\iota_{mn} \otimes \iota_{mn}"] 
			&\Lsp_m \times \Lsp_m \arrow[d,"\eta_{nm} \varphi_n \otimes \eta_{nm} \varphi_n"]  \\
		\Lsp_n \times \Lsp_n \arrow[r,"\iota_{mn} \otimes \iota_{mn}"] &\Lsp_m \times \Lsp_m
	\end{tikzcd}~\hspace{20pt}
	\begin{tikzcd}
		\Lsp_n \times \Lsp_n \arrow[d,swap,"\varphi_n \otimes \varphi_n"] \arrow[r,"\iota_{n} \otimes \iota_{n}"] 
			&\Lsp \times \Lsp \arrow[d,"\eta_{n} \varphi_n \otimes \eta_{n} \varphi_n"]  \\
		\Lsp_n \times \Lsp_n \arrow[r,"\iota_{n} \otimes \iota_{n}"] &\Lsp \times \Lsp
	\end{tikzcd}
\]

\begin{proof}[Proof of Lemma~\ref{lem:nice_homomorphisms}]

Fix $1 \le n \le m$ and $\varphi_n \in \Phi_n$. 

\begin{enumerate}
\item Let $\xi_m \in M(\Lsp_m \times \Lsp_m)$. We need to show that
\[
	\pi_{mn} \eta_{nm}\phi_n \xi_m = \varphi_n \pi_{mn} \xi_m.
\]
as elements of $M(\Lsp_n \times \Lsp_n)$. So take $A_n \in \Mcal_n$. Then
\begin{align*}
	\pi_{mn} \eta_{nm}\phi_n \xi_m (A_n) 
	&= \eta_{nm}\varphi_n \xi_m(\iota_{nm} \otimes \iota_{nm} A_n)\\
	&= \xi_m(\eta_{nm} \varphi_n^{-1} \otimes \eta_{nm} \varphi_n^{-1} (\iota_{nm} \otimes \iota_{nm} A_n))\\
	&= \xi_m(\iota_{nm} \otimes \iota_{nm}(\varphi_n^{-1} \otimes \varphi_n^{-1} A_n))\\
	&= \varphi_n \pi_{mn} \xi_m (A_n),
\end{align*}
where the crucial third step uses the commuting diagram in point~1 of Lemma~\ref{lem:nice_homomorphisms}.
\item Let $\xi \in M(\Lsp \times \Lsp)$. We now need to show that
\[
	\pi_n \eta_n \varphi_n \xi = \varphi_n \pi_n \xi,
\]
as elements of $M(\Lsp_n \times \Lsp_n)$. Again, take $A_n \in \Mcal_n$. Then
\begin{align*}
	\pi_n \eta_n \varphi_n \xi (A_n) 
	&= \eta_n \varphi_n \xi(\iota_n \otimes \iota_n A_n)\\
	&= \xi(\eta_n \varphi_n^{-1} \otimes \eta_n \varphi_n^{-1} (\iota_n \otimes \iota_n A_n))\\
	&= \xi(\iota_n \otimes \iota_n (\varphi_n^{-1} \otimes \varphi_n A_n)) \\
	&= \varphi_n \pi_n \xi (A_n),
\end{align*}
where we used the commuting diagram in point~2 of the lemma for the third step.
\end{enumerate}
\end{proof}

\section*{Acknowledgments}
We thank Peter Orbanz, Christian Borgs, Olav Kallenberg, Svante Janson, Cosma Shalizi, Remco van der Hofstad, Morgane Austern, and Matan Harel for useful discussions and suggestions. This work was supported by NSF Grant No.~CCF-2311160.


\end{document}